\documentclass[10pt, reqno]{article}
    \usepackage{amsfonts}
    \usepackage{fancyhdr}
    \usepackage{comment}
    \usepackage
    [
        letterpaper, top=2.25cm, bottom=2.5cm, left=2.2cm, right=2.2cm
    ]
    {geometry}
    \usepackage{times}
    \usepackage{amsmath}
    \usepackage{amssymb}
    \usepackage{graphicx}
    \usepackage{relsize}
    \usepackage{multicol}
    \usepackage{enumerate}
    \usepackage{tikz-cd}
    \usepackage[framemethod=tikz]{mdframed}
    \usepackage{amsthm}
    \usepackage{titlesec}
    \usepackage{scrextend}
    \usepackage{array}
    \usepackage[bottom,flushmargin,hang,multiple,stable]{footmisc}
    \usepackage{hyperref}
\theoremstyle{plain}
\newtheorem{thrm}{Theorem}[section]
\newtheorem{lmm}[thrm]{Lemma}
\newtheorem{crllr}[thrm]{Corollary}

\newtheorem{crtrn}[thrm]{Criterion}

\theoremstyle{definition}
\newtheorem{dfntn}[thrm]{Definition}
\newtheorem{cnjctr}[thrm]{Conjecture}
\newtheorem{xmpl}[thrm]{Example}
\newtheorem{prblm}[thrm]{Problem}

    \titleformat{\section}
        [block]
        {
            \centering\fontsize{11}{13}
        }
        {\thesection.}{.25em}
        {
            \selectfont\scshape
        }
    \titleformat{\subsection}
        {
            \fontsize{11}{13}
        }
        {\thesubsection}{0.25em}
        {
            \bfseries
        }
    \def\xCzero{{\rm C}^{0}}    
    \def\xCone{{\rm C}^{1}} 
    \def\xR{\mathbb{R}}
    
    \def\xCn#1{{\rm C}^#1}
    \def\@Rref#1{\hbox{\rm \ref{#1}}}
    \def\Rref#1{\@Rref{#1}}
    \def\xGL{\mathop{\rm GL\,}\nolimits}
    \renewcommand{\dot}[1]
    {
    \overset{\raisebox{0ex}{\scalebox{0.45}{$\bullet$}}}{#1}
    }
    \newcommand{\restr}[2]
    {
    \left.\kern-\nulldelimiterspace
    #1 
    \vphantom{\big|}
    \right|_{#2}
    }
    \newcommand\fnsep{\textsuperscript{,}}
    \emergencystretch=\maxdimen
    \binoppenalty=\maxdimen
    \relpenalty=\maxdimen
    
\begin{document}
    \title
    {
        \textsc
        {
            Feedback stabilization of nonlinear control systems by composition operators
        }
    }
    \author
    {
        {
        \renewcommand{\thefootnote}{\arabic{footnote}}
        {\small\scshape BRYCE A. CHRISTOPHERSON\footnotemark[1] , \enskip BORIS S. MORDUKHOVICH\footnotemark[2]\renewcommand{\thefootnote}{\fnsymbol{footnote}} \fnsep \footnotemark[2] \renewcommand{\thefootnote}{\fnsymbol{footnote}} \renewcommand{\thefootnote}{\arabic{footnote}}\enskip, \enskip FARHAD JAFARI\footnotemark[3] }
        }
    }
    \footnotetext[1]
    {
        University of Wyoming, Dept. of Mathematics \& Statistics, Laramie, WY, USA; \href{mailto:BChris19@uwyo.edu}{BChris19@uwyo.edu}
    }
    \footnotetext[2]
    {
        Wayne State University, Dept. of Mathematics, Detroit, MI, USA; \href{mailto:Boris@math.wayne.edu}{Boris@math.wayne.edu}
    }
    \renewcommand{\thefootnote}{\fnsymbol{footnote}}
    \footnotetext[2]
    {
        Research of Boris S. Mordukhovich was partly supported by the USA National Science Foundation under grants DMS-1512846 and DMS-1808978, by the USA Air Force Office of Scientific Research under grant \#15RT04, and by the Australian Research Council under Discovery Project DP-190100555.
    }
    \renewcommand{\thefootnote}{\arabic{footnote}}
    \footnotetext[3]
    {
        Emeritus Professor of Mathematics, University of Wyoming, Dept. of Mathematics \& Statistics, Laramie, WY, USA and Professor of Radiology, University of Minnesota, Minneapolis, MN, USA; \href{mailto:FJafari@uwyo.edu}{FJafari@uwyo.edu}
    }
    \date{}
    \maketitle 
        {\scshape\bfseries\noindent Abstract.} Feedback asymptotic stabilization of control systems is an important topic of control theory and applications. Broadly speaking, if the system $\dot{x} = f(x,u)$ is locally asymptotically stabilizable, then there exists a feedback control $u(x)$ ensuring the convergence to an equilibrium for any trajectory starting from a point sufficiently close to the equilibrium state. In this paper, we develop a reasonably natural and general composition operator approach to stabilizability. To begin with, we provide an extension of the classical Hautus lemma to the generalized context of composition operators and show that Brockett's theorem is still necessary for local asymptotic stabilizability in this generalized framework. Further, we employ a powerful version of the implicit function theorem--as given by Jittorntrum and Kumagai--to cover stabilization without differentiability requirements in this expanded context.  Employing the obtained characterizations, we establish relationships between stabilizability in the conventional sense and in the generalized composition operator sense.  This connection allows us to show that the \textit{stabilizability} of a control system is equivalent to the \textit{stability} of an associated system.  That is, we reduce the question of \textit{stabilizability} to that of \textit{stability}.
                
        \bigskip
        

\section{Introduction}\label{intro}
    In this paper, we consider autonomous control systems of the form
    \begin{align}
    \dot{x}=f(x,u), & \enskip t \geq 0. \label{control sys}
    \end{align}
    \noindent More specifically, we take a neighborhood of the origin $\mathcal{X} \times \mathcal{U} \subseteq \xR^n \times \xR^m$ and, unless otherwise stated, we will always assume that our function $f$ on the right-hand side of \eqref{control sys} satisfies both conditions $f(0,0)=0$ and $f \in \xCzero\left (\mathcal{X} \times \mathcal{U}, \xR^n\right )$, where $\xCn{k}(\mathcal{X}\times\mathcal{U},\xR^n)$ denotes the $n$-times continuously differentiable functions $\mathcal{X}\times \mathcal{U} \rightarrow \xR^n$ (adopting as well the convention that taking $k=-1$ denotes the set of all functions $\mathcal{X}\times \mathcal{U} \rightarrow \xR^n$ without any assumption of continuity).
    
    Mostly, we are interested in {\it stabilizing} these systems. That is (as in, e.g., \cite[Definition 10.11]{coron1}), we are looking at the following property.
    \begin{dfntn}[Local Asymptotic Stabilizability - Feedback Laws]\label{local asymptotic stabilizability definition}
    Given a control system of type \eqref{control sys}, we say the system is \textit{locally asymptotically stabilizable by means of feedback laws} if there exists a neighborhood of the origin $\mathcal{O}\subseteq \mathcal{X}$ and a feedback $u  \colon  \mathcal{O} \rightarrow \mathcal{U}$ such that the origin is a locally asymptotically stable equilibrium \cite[Definition 2.1]{byrnes} of the closed-loop system $\dot{x}=f\big (x,u(x) \big )$.
    \end{dfntn}
    For some commonly-used additional jargon, if the stability of the system in Definition~\Rref{local asymptotic stabilizability definition} can be made exponential by such a feedback control, we say that the system is \textit{locally exponentially stabilizable by means of feedback laws}.  If, in addition, $u(0)=0$, we say that \eqref{control sys} is locally asymptotically (resp. exponentially) stabilizable by means of \textit{stationary} feedback laws.
    
    For brevity, when we reference local asymptotic stability throughout the paper, as defined in Definition~\Rref{local asymptotic stabilizability definition}, we will \textit{always} assume that the feedback $u(x)$ satisfies the following additional criterion:
    \begin{crtrn}[Uniqueness Criterion - Feedback Stabilization]\label{uniqueness criterion}
    The state-feedback controller $u(x)$ is such that $f\big(x,u(x)\big)$ is continuous and $\dot{x}=f\big (x,u(x) \big )$ has a unique solution $x(t)$ for all $t \geq 0$ and all $x_0$ in a neighborhood of the origin.
    \end{crtrn}
    There is some variation in the literature as to whether Criterion~\Rref{uniqueness criterion} ought to be included as part of Definition~\Rref{local asymptotic stabilizability definition} as, e.g., Coron \cite{coron2} and Zabczyk \cite{zab} do.  There seems to be no consensus here:  On one hand, a number of results may be stated more cleanly when adopting this additional restriction.  On the other, this assumption is somewhat unattractive in its restrictiveness, since it is also possible to have non-unique trajectories which still converge in the desired fashion. We will take extra care regarding this point.  In particular, when we reference known results in Section~\Rref{preliminaries section} while discussing some background materials, we will note which ones would be false without this imposition.  Additionally, it is important to note that Criterion~\Rref{uniqueness criterion} \textit{does not} require that $u(x)$ must be continuous.   That is, it is possible for the composition of a continuous function and a discontinuous function (or even two discontinuous functions) to be continuous.  For example, let $g(x)=x^2$ and let $D:\xR \rightarrow \xR$ be such that $D(x)=1$ if $x$ is rational and $D(x)=-1$ if $x$ is irrational.  Then, while $D$ is discontinuous everywhere, both $g \circ D$ and $D \circ D$ are continuous everywhere.
    
    A variety of conditions describing whether the system \eqref{control sys} can be locally asymptotically stabilized by means of continuous feedback laws have been derived; see, e.g., \cite{artstein,brockett,byrnes,christopherson2019,coron2,coron1,gjkm,hermes,sontag1,sontag2,jammazi,onishchenko1, onishchenko2}.  However, fully characterizing whether or not an arbitrary system has this property has proven to be a formidable task. At present, we are not familiar with any non-Lyapunov-theoretic conditions, which are simultaneously necessary and sufficient, that allow us to determine the existence or nonexistence of such a controller.
\subsection{Entering the New Approach} \label{entering the new approach subsection}
    Stabilization via feedback has been recognized as a difficult topic in control theory. One reason for why it might be is that the topic itself could be difficult.  That is, information regarding stabilization could, in a colloquial sense, be `buried' too deep or `tangled up' too well in the structure of a system to be extracted with ease.  Another option is that we've just been phrasing the question in an unfortunate manner and have made things \textit{seem} much more difficult than necessary. One way it might happen, as argued here, is that we opted to ask one particularly messy question instead of asking two much simpler questions that, when combined, produce the original messy one.  
            
    More particularly, when we apply a feedback law to a control system of form \eqref{control sys}, all we're doing is sending $f(x,u)$ to $f(x,u(x))$. At the end of the day, this amounts to applying a \textit{composition operator} $T_h$ to the vector field $f$ inducing the dynamics of the system, i.e., $T_hf=f \circ h$ for some mapping $h \colon \mathcal{X} \rightarrow \mathcal{X} \times \mathcal{U}$.  A fairly large amount is known about composition operators, so it seems like the real sticking point in the case of stability is that we require the generating mapping $h$ of our composition operator to be of a \textit{very specific form}.  Namely, $h(x)=(x,u(x))$ for some state-feedback control $u(x)$.  
    
    Of course, relaxing this condition would change the problem entirely, and answering a different problem instead of the one originally asked doesn't constitute a solution by any means. But, all the same, suppose one \textit{did} decide to go away with the restriction that the mapping $h$ in our composition must be of the form $h(x)=(x,u(x))$. Then, if one were to solve the problem of stability with the very general class of $h(x)$, certainly every necessary condition for the broader case would also be necessary for the particular case of the original question.  Indeed, if we knew that a system was stabilizable by some composition operator $T_h$, all that would then remain is a question that is hardly about controls at all--i.e., when can we choose the stabilizing composition operator $T_h$ to have a symbol of the form $h(x)=(x,u(x))$? 
    
    Broadly speaking, this is what we will attempt to do here.  That is, our approach to stabilizability by feedback will consist of asking the following two questions: 
    \begin{prblm}\label{question 1}
    Under what conditions does there exist a composition operator $T_h$, with a stationary symbol $h$ of some desired smoothness, such that the origin of the ”closed-loop” system $\dot{x}=T_hf(x)$ is locally asymptotically stable?
    \end{prblm}
    \begin{prblm}\label{question 2}
    Given the existence of \textit{some} stabilizing composition operators $T_h$ resolving Problem~\Rref{question 1}, does there exist at least one that is a control?
    \end{prblm}
    As we will show (Theorem~\Rref{hautus for comp stab implies continuous right inverse} and Theorem~\Rref{coronconjecture}), Problem~\Rref{question 1} is answered in the affirmative when $f$ has a local section $\alpha:\xR^n \rightarrow \xR^n\times\xR^m$ of the desired smoothness near the equilibrium.  Likewise, we will show (Theorem~\Rref{stabilizable so that infinitely differentiable characterization}) that Problem~\Rref{question 2} is answered in the affirmative when the projection onto its first factor of this local section, i.e. $\textrm{proj}_1\circ\alpha:\xR^n \rightarrow \xR^n$, induces a diffeo/homeomorphism whose inverse, itself, yields a stable system $\dot{x}=(\textrm{proj}_1 \circ \alpha)^{-1}(x)$.  Moreover, the answer to Problem~\Rref{question 2} will tell us how to produce such a stabilizing control (indeed, all possible stabilizing controls) from a local section satisfying the required criteria.
    
    Let us detail the preliminary definitions in this approach more precisely before continuing.
\subsection{Composition Stabilizability}\label{composition stabilizability subsection}
    In one sense, we will somewhat {\it generalize} the notion of local asymptotic stabilizability and look at the way our generalization interacts with the conventional form of local asymptotic stabilizability. In another (ultimately, equivalent) sense, we consider the classical notion of feedback local asymptotic stabilizability and begin by \textit{restricting} ourselves to certain particularly convenient cases. This will seem to be less confusing after a bit of explanation that follows.
    
    The main idea behind the generalization offered below is to treat the application of a stabilizing control as the action of a {\it composition operator}.  That is, take $f \in \xCn{k}(\mathcal{X}\times\mathcal{U},\xR^n)$ and let's suppose that the system $\dot{x}=f(x,u)$ is locally asymptotically stabilized by the feedback law $u(x)$. According to the definition, it says that the system $\dot{x}=f\big (x,u(x) \big)$ is locally asymptotically stable around the origin. Taking $\mathcal{O} \subseteq \mathcal{X}$ to be a neighborhood of the origin and $h \in \mathrm{C}(\mathcal{O}, \mathcal{X}\times\mathcal{U})$ be the mapping $x \mapsto (x,u(x))$, we can view $f\big (x,u(x)\big)$ as $(T_hf)(x)$, where $T_h$ is the composition operator $T_hf:=f \circ h$. More precisely, we define the following property.
    \begin{dfntn}[Local Asymptotic Stabilizabilty - Composition Operators]\label{local asymptotic stabilizability definition - composition operators}
    Considering the control system \eqref{control sys} with $f \in \mathrm{C}(\mathcal{X}\times\mathcal{U},\xR^n)$ satisfying $f(0,0)=0$, we say the system is \textit{locally asymptotically stabilizable by means of a composition operator} if there exists a neighborhood of the origin $\mathcal{O}\subseteq \mathcal{X}$ and a composition operator $T_h$ with symbol $h \colon \mathcal{O} \rightarrow \mathcal{X}\times \mathcal{U}$ such that the origin is a locally asymptotically stable equilibrium of the system $\dot{x}=T_hf(x)$.
    \end{dfntn}
    If such a stabilizing composition operator $T_h$ can be chosen with a symbol $h \in \xCn{k}(\mathcal{O},\mathcal{X}\times\mathcal{U})$ as $ -1 \leq k \leq \infty$, we say that the system \eqref{control sys} is {\it locally asymptotically stabilizable by means of a composition operator with a $\xCn{k}$ symbol}. If in addition, $h(0)=(0,0)$, we say the symbol is \textit{stationary}.  Additionally, as with the case of stabilizing feedbacks, we make an analogous uniqueness assumption to that of Criterion~\Rref{uniqueness criterion}.
    \begin{crtrn}[Uniqueness Criterion - Composition Stabilization]\label{uniqueness criterion for composition operators}
    The composition operator $T_h$ is such that $T_hf$ is continuous and $\dot{x}=T_hf(x)$ has a unique solution $x(t)$ for all $t \geq 0$ and all $x_0$ in a neighborhood of the origin.
    \end{crtrn}
    Asking whether or not there exists a locally asymptotically stabilizing \textit{composition operator} certainly encompasses the question of whether or not there exists a locally asymptotically stabilizing \textit{feedback law}--all one needs to do is simply take $h(x)=\big (x,u(x) \big)$ for any such stabilizing feedback law $u(x)$.  However, this is properly a weaker property, in the sense that it allows for compositions in the {\it state space domain} as well as the control domain.  More succinctly, we are just treating \textit{everything} as a control in this generalization. To highlight it, let's look at the following straightforward example.
    \begin{xmpl}\label{composition stabilizability example}
    Let $f(x,u):=x$. Since the application of any control has no effect on the dynamics, the corresponding system \eqref{control sys} is not locally asymptotically stabilizable by means of feedback laws.  However, \eqref{control sys} {\it is} locally asymptotically stabilizable by means of a composition operator (in fact, by one with a continuous symbol). Indeed, take, e.g., $h(x):=(-x,u(x))$ for any continuous feedback law $u(x)$. Then the system
    $$
    \dot{x}=(T_hf)(x)=f(h(x))=-x
    $$
    has a locally asymptotically stable equilibrium at the origin. In fact, it is easy to see that the origin is a globally exponentially stable equilibrium.
    \end{xmpl}
    As mentioned above, it is also reasonable to consider stabilizability by composition operators as a mere subcase of stabilizability by feedback laws, since treating \textit{everything} as a control is really the same thing as considering local asymptotic stabilizability for the restrictive subclass of systems which \textit{only} depend on controls (i.e., systems of form \eqref{control sys} that depend trivially on the state variable).\vspace*{0.05in}
    
    In this paper, we will show that this natural (albeit, much weaker) generalization of stabilizability (equivalently, conventional stabilizability for the restrictive subclass of systems) has some interesting connections to the problem of stabilizability by feedback laws given in Definition~\Rref{local asymptotic stabilizability definition}. To do so, we will  build up this approach and detail how many of the classical theorems on stabilizability have much cleaner forms in the composition operator context. We will show that extracting properties about the composition operators which stabilize a given system of form \eqref{control sys} tends to be, in general, relatively easy.  Additionally, we will make a case for the position that this is the correct approach to view the problem of feedback stabilizability.  In support of this stance, we will completely resolve Problem~\Rref{question 2} and, to some extent, Problem~\Rref{question 1}.
\subsection{Main Results of the Paper}\label{results subsection}
    The contribution of this paper is to reduce the question of \textit{stabilizability} to that of \textit{stability}.  That is, we will take the broader class of control systems of the form \eqref{control sys} and reduce the question of their \textit{stabilizability} to the question of the \textit{stability} of an associated system.  This departs substantially from the existing Lyapunov-theoretic solution to the problem of stabilizability for general nonlinear systems \eqref{control sys}--wherein one must determine the existence or nonexistence of a so-called control Lyapunov function, though it is not always clear how one might go about doing this.  Instead, we present a characterization of stabilizability entirely in terms of the stability of a system derived from a local section of the vector field $f$ in \eqref{control sys}.  Further, this does not only provide a characterization of local exponential and asymptotic stabilizability by feedback laws, but also produces a characterization of all possible stabilizing feedback laws.  More interestingly perhaps, this approach provides a somewhat unexpected link between the topic of stabilization and the existence of local sections. 
    
    To allow for concise wording, we first need a definition:
    \begin{dfntn}[Sections, Local Sections]\label{global section}
    Given topological spaces $X$, $Y$ and a function $F \colon X \rightarrow Y$, a \textit{global section} of $F$ is a continuous right inverse of $F$.  That is, a mapping $\sigma \colon Y \rightarrow X$ is a section of $F$ if $\sigma$ is continuous and $F \circ \sigma = \textrm{id}_{Y}$.  A \textit{local section} of $F$ is a continuous mapping $\sigma \colon \mathcal{O}\rightarrow X$ for some open set $\mathcal{O} \subseteq Y$ such that $F \circ \sigma=\iota \circ \textrm{id}_{\mathcal{O}}$, where $\iota$ denotes the inclusion map $\iota \colon \mathcal{O} \rightarrow Y$.  If $y_0 \in Y$, we say that $F$ has a \textit{local section near $y_0$} if there exists a local section $\sigma:\mathcal{O}\rightarrow X$ of $F$ such that $\mathcal{O}$ is an open neighborhood of $y_0$.
    \end{dfntn}
    For the remainder of this subsection, we will \textit{always} assume that Criterion~\Rref{uniqueness criterion} and Criterion~\Rref{uniqueness criterion for composition operators} hold whenever we reference stabilizability by feedback laws or stabilizability by composition operators.  Our results are as follows:
    
    \medskip
    
    First, we address the case of local asymptotic stabilizability by composition operators.  We begin by addressing the case of local \textit{exponential} stabilizability of smooth systems by composition operators with smooth stationary symbols.
    
    \smallskip
    
    \noindent\textbf{Theorem~\ref{hautus for comp stab implies continuous right inverse}}
    \textit{The system \eqref{control sys} with $f \in \xCn{k}(\mathcal{X}\times\mathcal{U},\xR^n)$ for $k \geq 1$ is locally exponentially stabilizable by means of a composition operator $T_h$ with a $\xCn{k}$ stationary symbol if and only if $f$ has a $\xCn{k}$ stationary local section near the origin.}
    
    \smallskip
    
    For stabilization by composition operators with merely continuous stationary symbols, the case is more interesting.
    
    \smallskip
    
    \noindent\textbf{Theorem~\ref{coronconjecture}}
    \textit{The system \eqref{control sys} is locally asymptotically stabilizable by means of a composition operator $T_h$ with a continuous stationary symbol if and only if $f$ has a stationary local section near the origin.  Moreover, if \eqref{control sys} is locally asymptotically stabilizable by such a composition operator, then this stability can always be made exponential.}
    
    \smallskip
    
    
    This produces an interesting gap, which is not resolved here.  Namely, exponential and asymptotic local stabilization by composition operators with continuous stationary symbols coincide, while they may not in the case where $k > 0$.  In particular, for local asymptotic stabilizability by a composition operator with an $\xCn{k}$ stationary symbol to be possible when local exponential stabilizability by the same is not, this requires that all of the local sections near the origin of the vector field $f$ in \eqref{control sys} must fail to be differentiable.  However, while this is necessary, it is unclear what conditions are sufficient to provide the existence of such a stabilizing composition operator.
    
    We make a connection to Brockett's theorem as well, providing a variety of composition-operator stabilization characterizations of this well-known property.
    
    \smallskip
    
    \noindent\textbf{Theorem \ref{discontinuous composition stabilizability theorem}}
    \textit{If the vector field $f$ in \eqref{control sys} satisfies Brockett's condition, then \eqref{control sys} is locally exponentially stabilizable by means of a composition operator with a (potentially, discontinuous) stationary symbol.}
    
    \smallskip
    
    \noindent\textbf{Corollary \ref{isolated equilibrium - brockett's necessary and sufficient theorem}}
    \textit{Suppose the system \eqref{control sys} has the origin as an isolated equilibrium.  Then, \eqref{control sys} is locally asymptotically stabilizable by means of a composition operator $T_h$ with a stationary symbol which is continuous at the origin if and only if the vector field $f$ in \eqref{control sys} satisfies Brockett's condition.}
    
    \smallskip
    
    Using the tools developed for the case of stabilization by composition operators, we are able to give a necessary and sufficient condition for the existence of a stabilizing feedback laws.  Unlike the case of composition operators, this result leaves no remaining gap.
    
    \smallskip
    
    \noindent\textbf{Theorem~\ref{stabilizable so that infinitely differentiable characterization}}
    \textit{For $k \geq 0$, the system \eqref{control sys} is locally asymptotically (resp. exponentially) stabilizable by means of a $\xCn{k}$ stationary feedback law $u$ if and only if there exists a $\xCn{k}$ stationary local section $\alpha$ of $f$ near the origin such that $\alpha_1 := \textrm{proj}_1\circ \alpha$ is a homeomorphism (resp. diffeomorphism) and $\dot{x}=\alpha_1^{-1}(x)$ is locally asymptotically (resp. exponentially) stable.  Moreover, every such control $u$ is of the form $u =\textrm{proj}_2\circ \alpha \circ (\textrm{proj}_1 \circ \alpha)^{-1}$ for some such $\alpha$ satisfying the above.}
    
    \smallskip
    

\section{Preliminaries}\label{preliminaries section}
    In this section, we begin by briefly recalling some known results regarding stabilization by feedback laws and detailing a short taxonomy of the obstacles remaining between what is known and a complete characterization of local asymptotic stabilizability.
\subsection{Hautus Lemma and Related Results}\label{hautus lemma and related results subsection}
    As mentioned previously, when we reference local asymptotic stability throughout the paper, we will \textit{always} assume that the feedback $u(x)$ satisfies Criterion~\Rref{uniqueness criterion}.  In this section, however, we will be careful to also note which ones would be false without this imposition.
    
    As mentioned previously, a large number of partial characterizations of stabilizability do exist, and some of the more useful elementary ones are worth mentioning quickly. First, let us briefly recall some conventional notation. Given a mapping $f \in \xCone(\mathcal{X} \times \mathcal{U},\xR^n)$ as in \eqref{control sys}, we denote the Jacobian of $f$ by $J_f$ and the partial Jacobian matrices of $f$ at $(0,0)$ by
    \begin{align}
    A_f := \restr{\frac{\partial f}{\partial x}}{(0,0)} \enskip \textrm{and} \enskip B_f := \restr{\frac{\partial f}{\partial u}}{(0,0)}. \label{AB def}
    \end{align}
    \noindent e.g. so that, as a block matrix, $\restr{J_f}{(0,0)} = \begin{bmatrix}A_f & B_f\end{bmatrix}$.  Given a linear operator $T\colon\xR^n\to\xR^n$, denote its {\it spectrum} by $\Lambda(T)$ and the subset of $\Lambda(T)$ consisting of the eigenvalues with {\it nonnegative real part} by
    $$
    \Lambda_+(T) := \left \{\lambda \in \Lambda(T)\;\big|\;\textrm{Re}(\lambda)\geq 0 \right \}.
    $$
    \noindent Next we recount the celebrated Hautus lemma needed below.
    \begin{lmm}[Hautus]\label{hautus lemma}
    Given an $n \times n$ matrix $A$ and an $n \times m$ matrix $B$, the linear system $\dot{x} = A x + B u$ is locally exponentially stabilizable if and only if for all $\lambda \in \Lambda_+(A)$ it holds that
    $$
    \textrm{rank}\left[\begin{array}{c|c}\lambda I - A & B \end{array} \right] = n.
    $$ 
    \end{lmm}
    There is a similar result to the Hautus lemma, which applies to the linearization of a system like that given in \eqref{control sys}. That is, using the notation of \eqref{AB def}, it applies to the linearized system
    \begin{align}
    \dot{x} = A_f x + B_f u, \enskip t \geq 0. \label{linear control sys}
    \end{align} 
    This nonlinear analogue of the Hautus lemma via linearization was first proved by Zabczyk in \cite{zab}, though he mentioned that it was likely known as folklore for some time before a published proof was available.
    \begin{thrm}[Zabczyk]\label{zabczyk} The control system \eqref{control sys} with $f \in\xCone(\mathcal{X}\times\mathcal{U},\xR^n)$ is locally exponentially stabilizable by means of $\xCone$ stationary feedback laws if and only if the linearized system \eqref{linear control sys} is locally exponentially stabilizable.  Moreover, if the linearized system \eqref{linear control sys} is locally exponentially stabilizable, then the linear feedback law stabilizing \eqref{linear control sys} is a locally exponentially stabilizing feedback law for the original system \eqref{control sys}.
    \end{thrm}
    Combining Hautus' lemma and Zabczyk's theorem yields the following consequence.
    \begin{crllr}[Hautus-Zabczyk] \label{hautus-zabczyk}
    The control system \eqref{control sys} with $f\in\xCone(\mathcal{X}\times\mathcal{U},\xR^n)$ is locally exponentially stabilizable by means of $\xCone$ stationary feedback laws if and only if we have for all $\lambda \in \Lambda_+(A_f)$ that
    $$
    \textrm{rank}\left [\begin{array}{c|c}\lambda I-A_f & B_f \end{array}\right] = n.
    $$
    \end{crllr}
    The most famous result regarding asymptotic stabilizability is probably a remarkable and easily formulated necessary condition given by Brockett \cite{brockett}. We discuss it in the next subsection.
\subsection{Brockett's Theorem}
\label{brocketts theorem subsection}
    Brockett gave in \cite{brockett} a necessary condition for feedback asymptotic stabilizability of nonlinear systems, which has attained a great attention in control theory. This condition constitutes (with the proof based on topological degree theory) that $f$ must satisfy a certain `local openness' property if such a controller exists.
    \begin{dfntn}[Locally Open]\label{open}
    A mapping $f \colon \xR^\ell\rightarrow \xR^n$ is said to be {\it open at a point} $\bar{z}\in\xR^\ell$ if we have $f(\bar{z}) \in \textrm{int}f(\mathcal{O})$ for any neighborhood $\mathcal{O}$ of $\bar{z}$.
    \end{dfntn}
     The origin of this property goes back to the classical Banach-Schauder open mapping theorem, which is one of the central results of functional analysis.
       
     \begin{thrm}[Brockett's Theorem] \label{brocketts} 
     If the system \eqref{control sys} with $f\in\xCone(\mathcal{X}\times\mathcal{U},\xR^n)$ is locally asymptotically stabilizable by means of stationary $\xCone$ feedback laws, then it is necessary that $f$ is open at $(0,0)$.
     \end{thrm}
        
     Stated a bit differently, Brockett's theorem says that any system \eqref{control sys}, which is asymptotically stabilizable by means of continuously differentiable feedback laws, has a solution to $f(x,u)=y$ for all $\|y\|$ sufficiently small. More compactly, Brockett's condition says that such a system must be `locally surjective' on any neighborhood of the origin, i.e., $f$ is surjective onto a neighborhood of the origin in $\xR^n$ when restricted to some appropriate neighborhood of the origin in $\xR^n \times \xR^m$.  For intuition, the `archetypal' example of a mapping that satisfies this condition would be--in the spirit of the inverse function theorem--any such $f$ that has a Jacobian of full row-rank at the origin.  To contrast, while having a full row-rank Jacobian at the origin is certainly {\it sufficient}, it is definitely not {\it necessary}; take, for example, $f(x):=x^3$. 
      
     Strong attempts have been made to extend Brockett's theorem in some or another form and to `close the gap' between necessity and sufficiency when it comes to characterizing stabilizable systems. There are a large number of results like that, accounting for a similarly large number of relatively-involved distinctions regarding the continuity and differentiability classes of stabilizing controls--for example, by Coron in \cite{coron1}, by Sontag in \cite{sontag0,sontag2}, by Zabczyk in \cite{zab}, and many others. For an easy reference on a fairly large portion of the subject (and a truly excellent treatment of the topic overall), Byrnes \cite{byrnes} provides nice, streamlined proofs of both Coron and Zabczyk's results among other developments.\vspace*{0.05in}
     
     The following extension of Brockett's theorem from smooth to continuous system by means of continuous vs.\ smooth feedback laws is given by Coron \cite{coron2} with the usage of a result by Zabczyk from \cite{zab}.
     
     \begin{thrm}[Brockett's Theorem - Continuous Extension] \label{brocketts - continuous extension} If the system \eqref{control sys} is locally asymptotically stabilizable by means of continuous stationary feedback laws satisfying Criterion~\Rref{uniqueness criterion}, then it is necessary that the mapping $f$ is open at $(0,0)$.
    \end{thrm}
    Observe that it is \textit{not} known whether or not Brockett's condition is required for local asymptotic stabilizability in the more general case of merely continuous feedback laws \textit{without} the assumption of Criterion~\Rref{uniqueness criterion} regarding the uniqueness of locally defined solutions. That being said, it is worth mentioning (as Sontag notes in \cite[Section 5.8, 5.9]{sontag2} and \cite[p. 554]{sontag0}) that continuous feedback laws can often be `smoothed out' away from the origin. Thus, combined with Sontag's proof in \cite[p. 560]{sontag0} that stabilization via continuous feedback laws which are locally Lipschitz away from the origin yields Brockett's condition (without making the assumption of uniqueness of trajectories), this does lend some support to the {\em conjecture} that Brockett's theorem may be necessary in the general case of Definition~\Rref{local asymptotic stabilizability definition} sans the strong requirement of unique trajectories.
    
    To conclude this short overview of known results in the area particularly related to Brockett's theorem, let us mention a new {\em variational approach} to feedback stabilizability of nonlinear control systems that was initiated in \cite{gjkm} and then further developed in our paper \cite{christopherson2019}. Though variational analytic connections will not be our focus here, we should note that some results in \cite{christopherson2019} were reached by arguments involving the use of composition operators. Particularly, \cite[Theorem 9]{christopherson2019} gives a condition (less restrictive than what is required by Corollary~\Rref{hautus-zabczyk}) which ensures the existence of stabilizing controls which are $\xCone$ on only a \textit{punctured} neighborhood of the origin (that is, $\mathcal{N}\setminus \left\{(0,0)\right\}$ for some neighborhood of the origin $\mathcal{N}$).  This extends results like Corollary~\Rref{hautus-zabczyk} to cover some instances of the case where locally asymptotically stabilization of \eqref{control sys} is only possible by \textit{continuous} (and not $\xCone$), stationary feedback laws.  The techniques employed in \cite{christopherson2019} seem to suggest that further applications of composition operators--utilized more adroitly--might allow for further progress in this direction.  This constitutes, in large part, the motivation for our approach herein.

\subsection{Outline}\label{outline subsection}
    The rest of the paper is organized as follows. Section~\Rref{hautus lemma for comp op section} is devoted to developing the composition counterpart of the extended Hautus-Zabczyk result in the novel context. In Section~\Rref{Brockett type}, we derive new results for composition operators under Brockett's openness condition. To do so, Subsection~\Rref{Brockett's Condition and Local Quotient Maps subsection} establishes some subtle properties of local quotient maps under the openness condition.  In Subsection~\Rref{brocketts theorem for comp ops subsection} we show that Brockett's theorem is still necessary for this generalization of stabilizability and, in some cases, sufficient. In Section~\Rref{a complete characterization of composition exponential stabilizability section}, we solve the problem of exponential stabilization by means of composition operators.  In the process, we will also solve a portion of the problem of asymptotic stabilization by composition operators.
    
    In Section~\Rref{kumagai ift section chapter 3}, we return to the case of stabilization via feedback laws.  We produce a characterization of local asymptotic and exponential stabilizability by feedbacks, and characterize what stabilizing controls must be like.  The conclusion, Section~\Rref{conclusion section chapter 3}, summarizes the main developments of the paper and discusses some open problems for future research.
\section{Generalized Hautus Lemma for Composition Operators}\label{hautus lemma for comp op section}
    Let's begin by seeing how a well-known feedback stabilizability result carries over to the composition stabilizability setting. In the spirit of the Hautus-Zabczyk result (i.e., Corollary~\Rref{hautus-zabczyk}) regarding the local exponential stabilizability of a system by $\xCn{k}$ feedback laws, we establish an analogous criterion for the local exponential stabilizability of a system by means of composition operators with a $\xCn{k}$ symbol.
            
    The next result in this subsection was proven previously by the authors in \cite{christopherson2019} (see, e.g. \cite[Lemma 5]{christopherson2019}), although the context was not so explicitly stated. We opt to include a modified (and much simpler) version of the proof here. 
    \begin{crllr}[Hautus-Zabczyk - Composition Stabilizability]\label{Hautus for comp stab}
    The system \eqref{control sys} with $f \in \xCone(\mathcal{X}\times\mathcal{U},\xR^n)$ is locally exponentially stabilizable by a composition operator with a $\xCone$ stationary symbol if and only if 
    \begin{equation}
    \textrm{rank} \left ( \restr{J_f}{(0,0)} \right ) = n. \label{jac rank condition}
    \end{equation}
    \end{crllr}
    \begin{proof}
    Writing $w = (x,u)$, consider the system $\dot{y}=F(y,w) := f(w)$ with trivial dependence on the state variable $y$.  Clearly, stabilization of this system by feedback laws is equivalent to the stabilization of \eqref{control sys} by composition operators.  As such, Corollary~\Rref{hautus-zabczyk} gives the result. 
    \end{proof}
    This result isn't surprising at all--it is just a sub-case of that given in Corollary~\Rref{hautus-zabczyk}.  Indeed, the condition \eqref{jac rank condition} must happen in the case of stabilization by feedback laws as well: suppose that \eqref{control sys} is locally exponentially stabilizable by means of continuously differentiable feedback laws. Then there are two possibilities: either $A_f$ is full rank, or it is not. In the first case, if $A_f$ is of full rank, then $\restr{J_f}{(0,0)}$ certainly has full row-rank. In the second case, if $A_f$ is not of full rank, then $\lambda=0$ is certainly an eigenvalue of $A_f$ with nonnegative real part, and thus the Hautus lemma dictates that we have 
     $$
     \textrm{rank}\left[\begin{array}{c|c}\lambda I-A_f & B_f\end{array}\right] = \textrm{rank}\left( \restr{J_f}{(0,0)} \right) = n.
     $$
     
     The result of Corollary~\Rref{Hautus for comp stab} serves as good motivation for what we will proceed to do in much of the remainder of the paper.  Namely, notice that if \eqref{jac rank condition} holds, then $\restr{J_f}{(0,0)}$ has a \textit{continuous right inverse}--namely, the Moore-Penrose pseudoinverse $\left(\restr{J_f}{(0,0)}\right)^+$. In fact, this holds more generally.  
     
     In the following theorem, we show that the existence of a $\xCn{k}$ stationary local section near zero is equivalent to exponential stabilizability of \eqref{control sys} by means of a $\xCn{k}$ stationary composition operator.
     \begin{thrm}[Composition Stabilizability and Local Sections - $\xCn{k}$ Version]\label{hautus for comp stab implies continuous right inverse}
     The system \eqref{control sys} with $f \in \xCn{k}(\mathcal{X}\times\mathcal{U},\xR^n)$, where $k\geq 1$, is locally exponentially stabilizable by a composition operator with a $\xCn{k}$ stationary symbol if and only if $f$ has a  stationary $\xCn{k}$ local section at the origin.
     \end{thrm}
     \begin{proof}
     The reverse direction is obvious:  Suppose $f$ has a $\xCn{k}$ right inverse $\alpha$ defined on some open neighborhood $\mathcal{O}\subseteq \xR^n$ of the origin.  Take any system $\dot{x}=g(x)$ for which the origin is a locally exponentially stable equilibrium and $g$ is $\xCn{k}$.  Let $\iota:\mathcal{O} \rightarrow \xR^n$ denote the inclusion map and let $\tilde{g}:\mathcal{N} \rightarrow \mathcal{O}$ be such that $\iota \circ \tilde{g} = \restr{g}{\mathcal{N}}$ for some neighborhood of the origin $\mathcal{N}$ (i.e. so that $g(\mathcal{N}) \subseteq \mathcal{O}$, as is always possible by continuity).  Choosing $T_h$ to be the composition operator with symbol $h=\alpha \circ \tilde{g}$ so that 
     $$
     T_hf= f \circ \alpha \circ \tilde{g} = \iota \circ \tilde{g}= \restr{g}{\mathcal{N}},
     $$
     it is clear that $T_h$ locally exponentially stabilizes \eqref{control sys}.  
     
     In the forward direction, assume that the system \eqref{control sys} is locally exponentially stabilizable by means of a composition operator $T_h$ with a $\xCn{k}$ stationary symbol.  Then, the system $\dot{x}=T_hf(x)$ is locally exponentially stabilizable by means of the composition operator $T_{x}$, which certainly has a $\xCone$ stationary symbol.  By Corollary~\Rref{Hautus for comp stab}, $\restr{J_{T_hf}}{0}$ has full row-rank and the function $F(x,y) := T_hf(y) - x$ satisfies the conditions of the Implicit Function Theorem.  So, there exists an open neighborhood of the origin $\mathcal{O} \subseteq \xR^n$ and a $\xCn{k}$ function $g \colon \mathcal{O} \rightarrow \xR^n$ such that $g(0)=0$ and $F\big(x,g(x)\big) = 0$ for all $x \in \mathcal{O}$.  Correspondingly, the function $\alpha \colon \mathcal{O}\rightarrow \mathcal{X}\times \mathcal{U}$ defined by $\alpha := h \circ g$ is $\xCn{k}$, stationary, and satisfies $(f \circ \alpha )(x) = x$ for all $x \in \mathcal{O}$.  That is, $\alpha$ is a $\xCn{k}$ local section of $f$ near the origin.
     \end{proof}
     The central thesis of this paper's approach to stabilization by composition operators is largely characterized by the result above.  That is, in the remainder of the paper, we will show that Theorem~\Rref{hautus for comp stab implies continuous right inverse} generalizes to characterize stabilizability by means of composition operators almost entirely.  To this aim, we will begin by showing that Brockett's theorem holds in the context of stabilization by composition operators in the next section.  After some prerequisite details are dealt with following this, we will show that the obvious non-differentiable analog of the result of Theorem~\Rref{hautus for comp stab implies continuous right inverse} holds as well.
\section{Extended Brockett Theorem and Composition Properties under Openness}\label{Brockett type}
    This section studies the role of the openness property of vector fields of control systems in stabilizability of such systems by means of composition operators. We begin in Subsection~\Rref{Brockett's Condition and Local Quotient Maps subsection} by first investigating the openness property of Brockett's theorem, yielding a connection to quotient maps that will be important throughout the remainder of this paper.  After this, in Subsection~\Rref{brocketts theorem for comp ops subsection}, we show that Brockett's openness property is not only necessary, but is also sufficient, for stabilization by composition operators with stationary symbols that are continuous at the origin provided the origin is an isolated equilibrium of \eqref{control sys}.
\subsection{Openness Condition and Local Quotient Maps}\label{Brockett's Condition and Local Quotient Maps subsection}
    The first lemma here showcases a useful fact about continuous surjections between arbitrary topological spaces satisfying the property given in Definition~\Rref{open} for some point $x_0$ in their domain.
    \begin{lmm}\label{brockett's local quotient property}
    Let $X$ and $Y$ be topological spaces. Fix $x_0 \in X $ and suppose that $F \colon X \rightarrow Y$ is a continuous surjection which is open at $x_0$. Then, for any topological space $Z$ and any $g \colon Y\rightarrow Z$, the mapping $g$ is continuous at $F(x_0)$ if and only if the composition $g\circ F$ is continuous at $x_0$.
    \end{lmm}
    \begin{proof}
    In the forward direction, suppose that $g$ is continuous at $F(x_0)$. The continuity of the composition $g\circ F \colon X \rightarrow Z$ at $x_0$ follows immediately, since continuity is preserved under compositions. That is, if $Z_0 \subseteq Z$ is a neighborhood of $g(F(x_0))$, then the continuity of $g$ at $F(x_0)$ ensures that $g^{-1}(Z_0)$ is a neighborhood of $F(x_0)$. Using the continuity of $F$ at $x_0$ tells us that the inverse image $F^{-1}(g^{-1}(Z_0))$ is a neighborhood of $x_0$.
            
    In the reverse direction, suppose that $g \circ F$ is continuous at $x_0$. Then, for any neighborhood $Z_0$ of $g(F(x_0))$, there exists a neighborhood  $\mathcal{O}\subseteq X$ of $x_0$ such that $\mathcal{O} \subseteq (g \circ F)^{-1}(Z_0)$. Observing that 
    $$
    \mathcal{O} \subseteq (g \circ F)^{-1}(Z_0) = F^{-1}\big(g^{-1}(Z_0)\big)
    $$
    and employing the surjectivity of $F$ imply that $F\big(F^{-1}(S)\big)=S$ for all subsets $S \subseteq X$. Therefore, it yields
    $$
    F(\mathcal{O}) \subseteq F\big(F^{-1}(g^{-1}(Z_0))\big) = g^{-1}(Z_0).
    $$   
    Since $F$ is open at $x_0$ and $\mathcal{O}$ is a neighborhood of $x_0$, we get that $F(\mathcal{O})$ is a neighborhood of $F(x_0)$. Remembering that $Z_0$ was an arbitrarily chosen neighborhood of $g(F(x_0))$, the continuity of $g$ at $F(x_0)$ follows.
    \end{proof}
    Lemma~\Rref{brockett's local quotient property} has an interesting interpretation in terms of {\em quotient maps}.
    \begin{dfntn}
        Let $X$ and $Y$ be topological spaces.  We say that a continuous function $f:X\rightarrow Y$  is a quotient map if $f$ is surjective and $f^{-1}(U) \subseteq X$ is open if and only if $U \subseteq Y$ is open
    \end{dfntn}
    Recall that quotient maps are characterized among continuous surjections as follows; see, e.g., \cite{willard2004general}:
    \begin{thrm}\label{characteristic property of quotient maps}
    Let $q \colon X \rightarrow Y$ be a continuous surjection.  The following are equivalent:
    \begin{enumerate}
    \item[\rm(i)] $q$ is a quotient map.
    \item[\rm(ii)] For any topological space $Z$ and any mapping $g \colon Y\rightarrow Z$, we have that $g$ is continuous if and only if the composition $g \circ q$ is continuous.
    \end{enumerate}
    \end{thrm}
    Similarly, to show that a continuous surjection $q$ is a quotient map, recall that it is sufficient (though not necessary) to verify that $q$ is either an \textit{open map} or a \textit{closed map}.  That is, if a continuous surjection is to be a quotient map, it is sufficient that it is open at \textit{every point in its domain}.  Essentially, it is the \textit{global} analogue to the \textit{local} version given in the assumptions of Lemma~\Rref{brockett's local quotient property}.
        
    In fact, the similarities between this and what we have in the Definition~\Rref{open} are really quite apparent. Namely, instead of a \textit{continuous surjection} which is \textit{open at every point}, we have a continuous surjection that is \textit{open at a selected point}.  In this sense, a continuous surjection which satisfies Definition~\Rref{open} must be something like a \textit{local} quotient map, i.e., a quotient map \textit{at a point}.  In light of the characteristic \textit{global} property of quotient maps (i.e., Theorem~\Rref{characteristic property of quotient maps}), it's not surprising that the \textit{local analogue} of a quotient map induces a \textit{local} analogue of this characteristic property (i.e., Lemma~\Rref{brockett's local quotient property}).\vspace*{0.05in}
    
    As it turns out, we can go further in some situations.  Recall that we say a map $F \colon X \rightarrow Y$ is \textit{proper} if $F^{-1}(K)$ is compact for every compact set $K \subseteq Y$.  Similarly, we say that a topological space $X$ is \textit{locally compact} if every point $x \in X$ has a compact neighborhood.  As is well-known, if $X$ is Hausdorff, this is equivalent to the statement that every neighborhood $N$ of a point $x \in X$ contains a compact neighborhood $K$ of $x$.  Finally, recall, as mentioned previously following Theorem~\Rref{characteristic property of quotient maps}, that if one desires to show that a mapping $F \colon X \rightarrow Y$ is a quotient map, then it is sufficient to show that $F$ is a closed, continuous surjection.
    \begin{lmm}\label{surjective restriction}
    Let $X$ and $Y$ be topological spaces and suppose that $F \colon X \rightarrow Y$ is continuous, $F(x_0)=y_0$, and $F$ is open at $x_0 \in X$. Then there exist an open neighborhood $Y_0 \subseteq Y$ of $y_0$, an open neighborhood $X_0\subseteq X$ of $x_0$, and a continuous surjection $\tilde{F} \colon X_0 \rightarrow Y_0$ which is open at $x_0$ and satisfies $\iota \circ \tilde{F} = \restr{F}{X_0}$, where $\iota \colon Y_0 \rightarrow Y$ denotes the inclusion map.  
    
    If, in addition, both $X$ and $Y$ are locally compact and $Y$ is Hausdorff, then we can instead take $X_0$ and $Y_0$ to both be compact and such that $\tilde{F}$ is a closed quotient map.
    \end{lmm}
    \begin{proof}
    Since $F$ is open at $x_0 \in X$, it follows that $F(x_0) \in \textrm{int}F(X_0)$ for any neighborhood $X_0 \subseteq X$ of $x_0$.  In particular, by $x_0 \in \textrm{int}(X)$ we have $F(x_0) \in \textrm{int}F(X)$. It allows us to choose an open neighborhood $Y_0$ of $F(x_0)=y_0 \in Y$ such that $Y_0 \subseteq F(X)$. Write $X_0:=F^{-1}(Y_0)$ and take the inclusion map $\iota \colon Y_0 \rightarrow Y$. Then $X_0$ is an open neighborhood of $x_0$ by the continuity of $F$, and the mapping $\tilde{F} \colon X_0\rightarrow Y_0$ defined by $\iota \circ \tilde{F}:=\restr{F}{X_0}$ is a surjection by construction. Since $\tilde{F}$ agrees with $F$ on $X_0$, the claimed openness and continuity follow as well.
    
    Now, let us also assume that $X$ is locally compact and that $Y$ is locally compact and Hausdorff.  Note that by the local compactness of $X$, there must exist a compact neighborhood $K\subseteq X$ of $x_0$.  If $Y$ is Hausdorff, then $\restr{F}{K} \colon K \rightarrow Y$ is a continuous map from a compact space to a Hausdorff space. By the Closed Map Lemma, $\restr{F}{K}$ is therefore both proper and closed.  By the openness of $F$ at $x_0$, we have that $F(x_0) \in \textrm{int}F(K)$, and there exists a neighborhood of $F(x_0)$ contained in $F(K)$.  As we have assumed that $Y$ is both locally compact and Hausdorff, we may take this to be a compact neighborhood $Y_0\subseteq F(K)$ of $F(x_0)$.  Since $\restr{F}{K}$ is proper and continuous, $X_0 := \restr{F}{K}^{-1}(Y_0) = K \cap F^{-1}(Y_0) \subseteq K$ is a compact neighborhood of $x_0$. Define the mapping $\tilde{F} \colon  X_0 \rightarrow Y_0$ by $\iota \circ \tilde{F} = \restr{\left(\restr{F}{K}\right)}{X_0} = \restr{F}{X_0}$, where $\iota \colon Y_0 \rightarrow Y$ is again the inclusion map.  Noting that a subspace of a Hausdorff space is itself a Hausdorff space, $Y_0$ is Hausdorff and the mapping $\tilde{F}$ is continuous map from a compact space to a Hausdorff space and a surjection by construction.  Hence, via another application of the Closed Map Lemma, $\tilde{F}$ is a proper, closed, continuous surjection and therefore, a closed quotient map.
    \end{proof}
    Every function factors as a surjection (onto its image) and an injection (of the image into the codomain), so a factorization of this form is nothing surprising.  What Lemma~\Rref{surjective restriction} really tells us, in essence, is that systems \eqref{control sys} satisfying Brockett's necessary condition factor in this way, but with \textit{neighborhoods of the origin} as the respective domain and codomain of the corresponding surjection and inclusion.  Moreover, systems \eqref{control sys} that satisfy Brockett's condition do not just have a \textit{local analogue} to the characteristic property of quotient maps, but factor as inclusions and quotient maps \textit{themselves} on some suitable restriction of their domain.
    
    Recall that any surjection $F \colon X \rightarrow Y$ necessarily has a right inverse $\alpha \colon Y \rightarrow X$, i.e. a mapping $\alpha$ such that $F \circ \alpha = \textrm{id}_{Y}$.  This gives us a nice, albeit obvious, result.  For clarity, there is \textit{no assumption} on the continuity of the symbol of the stabilizing composition operator in what follows.
    \begin{thrm}[Discontinuous Composition Stabilizability]\label{discontinuous composition stabilizability theorem}
    If the vector field $f$ in \eqref{control sys} satisfies Brockett's condition, then \eqref{control sys} is locally exponentially stabilizable by means of a composition operator with a stationary symbol.
    \end{thrm}
    \begin{proof}
    If \eqref{control sys} satisfies the necessary condition of Brockett's theorem, then $f$ is open at $(0,0)$.  As such, Lemma~\Rref{surjective restriction} implies that there exists a neighborhood of the origin $X_0 \subseteq \mathcal{X} \times \mathcal{U}$ and a neighborhood of the origin $Y_0 \subseteq \xR^n$ such that the mapping $\tilde{f} \colon X_0 \rightarrow Y_0$, defined by $\iota_1 \circ \tilde{f} = \restr{f}{X_0}$ via the inclusion map $\iota_1:Y_0 \rightarrow \xR^n$, is a surjection.  Hence, $\tilde{f}$ has a right inverse $\alpha \colon Y_0 \rightarrow X_0$, and as $\restr{f}{X_0}(0,0)=0$, we may certainly choose this right inverse so that it satisfies $\alpha(0)=(0,0)$.  Take any system $\dot{x}=g(x)$ for which the origin is a locally exponentially stable equilibrium and $g$ is continuous.  Let $\tilde{g}:\mathcal{N} \rightarrow Y_0$ be such that $\iota_1 \circ \tilde{g} = \restr{g}{\mathcal{N}}$ for some neighborhood of the origin $\mathcal{N}$ (i.e. so that $g(\mathcal{N}) \subseteq Y_0$, as is always possible by continuity).  Set $h=\iota_2 \circ \alpha \circ \tilde{g}$, where $\iota_2:X_0 \rightarrow X$ is the inclusion map, and notice that, since $\iota_1 \circ \tilde{f} = \restr{f}{X_0}$ and $f \circ \iota_2 = \restr{f}{X_0}$, we have 
    \begin{align*}
    T_hf &= f \circ \iota_2 \circ \alpha \circ \tilde{g} \\
    &= \restr{f}{X_0} \circ \alpha \circ \tilde{g} \\
    &= \iota_1 \circ \tilde{f} \circ \alpha \circ \tilde{g} = \iota_1 \circ \tilde{g} = \restr{g}{\mathcal{N}}.    
    \end{align*}
    Since $\dot{x}=g(x)$ has the origin as a locally exponentially stable equilibrium and satisfies $g(0)=0$, so does $\dot{x}=\restr{g}{\mathcal{N}}(x)$.  It then follows that $T_h$ is a composition operator with a stationary symbol which locally exponentially stabilizes \eqref{control sys}.
    \end{proof}
    To reiterate, we stress that there is \textit{absolutely no assurance of continuity} for the symbol $h$ of the stabilizing composition operator $T_h$.  Instead, Theorem~\Rref{discontinuous composition stabilizability theorem} simply highlights the fact that, while stabilization by composition operators is remarkably easy to obtain in general (on a theoretical level, at least), the real trick is obtaining \textit{continuity} for the stabilizing composition operator's symbol.  We will proceed towards this end with some more characterizations of systems satisfying Brockett's condition and the potential continuity of their restrictions' right inverses.
    
    Recall now that, given a mapping $F \colon  X \rightarrow Y$, we say a subset $S\subseteq X$ is \textit{$F$-saturated} if $S = F^{-1}(F(S))$.  As illustrated by the following lemma (the proof of which can be found in, e.g., \cite[Lemma 2.3]{nlab:saturated_subset}), $F$-saturated subsets of closed mappings behave nicely.
    \begin{lmm}\label{saturated subset lemma}
    Let $X$ and $Y$ be topological spaces, let $F \colon X \rightarrow Y$ be a closed map, and assume $\mathcal{C} \subseteq X$ is closed and $F$-saturated.  Then, for every open set $\mathcal{O} \subseteq X$ such that $\mathcal{C} \subseteq \mathcal{O}$, there exists an $F$-saturated open set $\mathcal{V} \subseteq X$ such that $\mathcal{C} \subseteq \mathcal{V} \subseteq \mathcal{O}$.
    \end{lmm}
    This immediately leads to the following result:
    \begin{lmm}\label{right inverse continuous at zero lemma}
    Let $X$ and $Y$ be locally compact and let $Y$ be Hausdorff.  Suppose that $F \colon X \rightarrow Y$ is continuous, $F(x_0)=y_0$, and $F$ is open at $x_0 \in X$.  Then, if $F(x) = y_0$ has only the trivial solution $x = x_0$ for all $x$ in some sufficiently small neighborhood of $x_0$, any right inverse of the closed quotient map $\tilde{F} \colon X_0 \rightarrow Y_0$ (guaranteed to exist by Lemma~\Rref{surjective restriction}) is continuous at $y_0$ whenever $X_0$ and $Y_0$ are chosen to be sufficiently small.
    \end{lmm}
    \begin{proof}
    By Lemma~\Rref{surjective restriction}, the existence of the closed quotient map follows immediately from our assumptions.  Noting that if $F(x) = y_0$ has only the trivial solution $x = x_0$ for all $x$ in some sufficiently small neighborhood of $X_0$, suppose we have chosen $X_0$ and $Y_0$ such that $\tilde{F}(x)=y_0$ has only the solution $x=x_0$ (as is always possible, viz. the proof of Lemma~\Rref{surjective restriction}).  Then, the subset $\mathcal{C} = \left\{x_0 \right\} \subseteq X_0$ satisfies 
    $$
    \tilde{F}^{-1}\big(\tilde{F}(\mathcal{C})\big) = \tilde{F}^{-1}\big(\tilde{F}(x_0)\big) = \tilde{F}^{-1}(y_0)=x_0
    $$ 
    and is therefore $\tilde{F}$-saturated.  Since $\tilde{F}$ is a surjection, there exists a (potentially non-unique) right inverse $\alpha \colon Y_0 \rightarrow X_0$ of $\tilde{F}$.  Taking $\mathcal{O}$ to be any neighborhood of $x_0$, notice that $\mathcal{O}$ must necessarily contain an open set containing $\mathcal{C}$.  Hence, by Lemma~\Rref{saturated subset lemma}, there exists an open neighborhood $\mathcal{N}$ of $x_0$ such that $\mathcal{N}$ is $\tilde{F}$-saturated and $\mathcal{C}\subseteq \mathcal{N} \subseteq \mathcal{O}$.  So, since $\mathcal{O} \supset \mathcal{N}$ and $\tilde{F}^{-1}\big(\tilde{F}(\mathcal{N})\big) = \mathcal{N}$, we have 
    $$
    \alpha^{-1}(\mathcal{O}) \supseteq \alpha^{-1}(\mathcal{N}) = \alpha^{-1}\big(\tilde{F}^{-1}\big(\tilde{F}(\mathcal{N})\big)\big) = (\tilde{F}\circ \alpha)^{-1}\big(\tilde{F}(\mathcal{N})\big) = (\textrm{id}_{Y_0})^{-1}\big(\tilde{F}(\mathcal{N})\big) = \tilde{F}(\mathcal{N}).
    $$
    By openness, $\tilde{F}(\mathcal{N})$ contains a neighborhood of $y_0$, and therefore $\alpha^{-1}(\mathcal{O})$ must as well for any neighborhood $\mathcal{O}$ of $\alpha(y_0)=x_0$.  That is, any right inverse $\alpha$ of $\tilde{F}$ is continuous at $y_0$.
    \end{proof}
    Essentially, Lemma~\Rref{right inverse continuous at zero lemma} just tells us that systems \eqref{control sys}, which satisfy Brockett's necessary condition and have the origin as an isolated equilibrium, \textit{automatically} have (viz. Theorem~\Rref{discontinuous composition stabilizability theorem}) locally asymptotically (and exponentially) stabilizing composition operators with symbols that are stationary and continuous at the origin.  This will lead nicely into the results of the following subsection.
\subsection{Brockett's Theorem in the Class of Composition Operators}\label{brocketts theorem for comp ops subsection}
    The next result provides a natural extension of Brockett's theorem to the class of stabilizing composition operators with continuous symbols.
    \begin{thrm}[Brockett's Theorem for Composition Stabilizability]\label{comp brockett}
    If the system \eqref{control sys} is locally asymptotically stabilizable by means of a composition operator with a continuous symbol, then it is necessary that the vector field $f(x,u)$ is open at $(0,0)$.
    \end{thrm}
    \begin{proof}
    Let $y \in \xR^n$ and write $w:=(x,u)\in \xR^n \times \xR^m$.  Set $\mathcal{W}:=\mathcal{X} \times \mathcal{U}$ and choose a neighborhood of the origin $\mathcal{Y}\subseteq \xR^n$. Let $F \colon \mathcal{Y}\times\mathcal{W} \rightarrow \xR^n$ be given by $F(y,w):=f(w)$; so $F$ is independent of $y$. If the system \eqref{control sys} is locally asymptotically stabilizable by means of a composition operator $T_h$ with a continuous symbol $h(\cdot)$, then the one in $\dot{y}=F(y,w)$ is locally asymptotically stabilizable by means of continuous feedback laws, namely $w(y):=h(y)$. Applying Theorem~\Rref{brocketts - continuous extension} tells us it must be the case that $0 \in \textrm{int}F(\mathcal{Y}_0 \times \mathcal{W}_0)=\textrm{int}f(\mathcal{W}_0)$ for any neighborhood of the origin $\mathcal{Y}_0 \times \mathcal{W}_0 \subseteq \mathcal{Y} \times \mathcal{W}$.  Since $\mathcal{W}_0$ is an arbitrary neighborhood of the origin in $\mathcal{X}\times\mathcal{U}$, we justify the result.
    \end{proof}
    The fact that Brockett's theorem holds when we treat \textit{everything}--that is, both the state variable $x$ and the control variable $u$ in the standard setting--as a control gives us a good perspective to understand why Brockett's condition is necessary for stabilizability by feedback laws, but far from sufficient. Specifically, straightforward cases like that of Example~\Rref{composition stabilizability example} demonstrate that stabilizability by a composition operator is a much weaker property than stabilizability by a feedback law. Yet Brockett's condition is necessary in either case.
    
    As it turns out, Brockett's condition remains necessary under even weaker conditions.  The following theorem presents one such situation.
    \begin{thrm}[Brockett's Theorem for Composition Stabilizability - Continuity at Zero]\label{brockett's necessary and sufficient theorem}
    If the system \eqref{control sys} is locally asymptotically stabilizable by means of a composition operator $T_h$ with a stationary symbol which is continuous at the origin, then it is necessary that the vector field $f$ in \eqref{control sys} satisfies Brockett's condition.
    \end{thrm}
    \begin{proof}
    If \eqref{control sys}  is locally asymptotically stabilizable by means of a composition operator $T_h$ with a stationary symbol $h$ which is continuous at the origin and is such that $T_hf$ is continuous, then the system $\dot{x}=T_hf(x)$ is locally asymptotically stabilizable by the composition operator $T_x$, which certainly has a continuous stationary symbol.  Hence, by Theorem~\Rref{comp brockett}, $T_hf$ must be open.  Since $h$ is assumed to be stationary and continuous at zero, it follows that for every neighborhood $\mathcal{V}$ of $h(0)=(0,0)$, there exists a neighborhood $\mathcal{O}$ of $0$ such that $h(\mathcal{O}) \subseteq \mathcal{V}$.  Correspondingly, by the openness of $T_hf = f \circ h$, there exists a neighborhood $\mathcal{N}$ of $0$ such that $\mathcal{N} \subseteq T_hf(\mathcal{O}) = f\big(h(\mathcal{O})\big) \subseteq f(\mathcal{V})$.  Therefore, $f$ is open at the origin.
    \end{proof}
    \begin{xmpl}
    Consider Brockett's famous integrator, the system \eqref{control sys} with $f \in \xCone(\xR^3\times\xR^2,\xR^3)$ given by
    $$
        f(x,u) = \begin{bmatrix} 
        u_1 \\
        u_2\\
        x_1u_2 - x_2u_1
        \end{bmatrix}.
    $$ 
    It is easy to see that $f$ fails to satisfy Brockett's condition as stated in Theorem~\Rref{brocketts}, as the $f$-image of any neighborhood of the origin fails to contain points of the form $(0,0,\epsilon)$ for $\epsilon \neq 0$.  Consequently, there can be no continuous stationary feedback law stabilizing this system.  In light of Theorem~\ref{brockett's necessary and sufficient theorem}, this deficiency can be extended and is, in fact, even stronger--not only can there be no continuous stationary stabilizing feedback, this continuity is impossible to achieve even just at the equilibrium.  Moreover, neither of these deficiencies can be remedied by moving to the larger context of stabilizing composition operators.
    \end{xmpl}
    After some consideration, Theorem~\Rref{brockett's necessary and sufficient theorem} is relatively natural, as openness at the origin is a local property and continuity at the origin is the local analogue to continuity.  In light of Theorem~\Rref{brockett's necessary and sufficient theorem}, Lemma~\Rref{right inverse continuous at zero lemma}, and Theorem~\Rref{discontinuous composition stabilizability theorem}, we pose the following conjecture:
    \begin{cnjctr}\label{brockett's sufficient conjecture}
    If the vector field $f$ in \eqref{control sys} satisfies Brockett's condition, then \eqref{control sys} is locally exponentially stabilizable by means of a composition operator $T_h$ with a stationary symbol which is continuous at the origin.
    \end{cnjctr}
    If Conjecture~\Rref{brockett's sufficient conjecture} is true, then Brockett's condition is (by Theorem~\Rref{discontinuous composition stabilizability theorem}) necessary and sufficient for stabilizability by a composition operator with a stationary symbol that is continuous at the origin.  We suspect that this could possibly be the case, but can also imagine that there may be pathologies preventing such a neat characterization (e.g. it might be necessary to impose that the fiber of $f$ over the origin to be acyclic as well, or some similarly nuanced condition).  In support of this conjecture, however, we do have the following immediate result via Theorem~\Rref{brockett's necessary and sufficient theorem} and Theorem~\Rref{right inverse continuous at zero lemma}:
    \begin{crllr}[Brockett's Theorem for Composition Stabilizability - Continuity at Zero]\label{isolated equilibrium - brockett's necessary and sufficient theorem}
    Suppose the system \eqref{control sys} has the origin as an isolated equilibrium.  Then, \eqref{control sys} is locally asymptotically stabilizable by means of a composition operator $T_h$ with a stationary symbol which is continuous at the origin if and only if the vector field $f$ in \eqref{control sys} satisfies Brockett's condition.
    \end{crllr}
    Having now touched on Brockett's condition in the context of composition stabilization, we are now ready to proceed to composition stabilization for non-differentiable systems.
\section{Composition Stabilizability without Differentiability}
\label{a complete characterization of composition exponential stabilizability section}
    To proceed, we must develop a method to characterize stabilizability by a composition operator with a continuous, but not continuously differentiable, symbol.  To do this, it will be helpful first to state a much more general version of the implicit function theorem that allows us to deal with this non-differentiability a bit better. Such a version of the (generalized) implicit function theorem was obtained by Kumagai in \cite{kumagai}, which is a refinement of the origin result given by Jittorntrum in \cite{jittorntrum}.
    
    The major goal of this section is to answer this remaining question and solve the problem of exponential stabilization by means of composition operators.  In the process, we solve a portion of the problem of asymptotic stabilization by composition operators as well.  To do so, we will begin in Subsection~\Rref{kumagai jittorntrum and composition stabilizability subsection chapter 3} by applying the Kumagai-Jittorntrum implicit function theorem judiciously.  
\subsection{Composition Stabilizability via Implicit Functions}\label{kumagai jittorntrum and composition stabilizability subsection chapter 3}
    We start this section with formulating the aforementioned Kumagai-Jittorntrum theorem .
    
    \begin{thrm}[Kumagai-Jittorntrum Implicit Function Theorem]\label{kumagai jittorntrum implicit function theorem}
    Let $F \colon  \xR^k \times \xR^\ell \rightarrow \xR^k$ be a continuous mapping, and let $F(a_0,b_0)=0$.  Then the following statements are equivalent:
    \begin{enumerate}
    \item[\rm(i)] There exists open neighborhoods $A_0 \subseteq \xR^k$ and $B_0 \subseteq \xR^\ell$ of $a_0$ and $b_0$, respectively, such that for all $b \in B_0$ the equation $F(a,b)=0$ has a unique solution $a=w(b)$ for some continuous mapping $w \colon  B_0 \rightarrow A_0$.
    \item[\rm(ii)] There exist open neighborhoods $A \subseteq \xR^k$ and $B \subseteq \xR^\ell$ of $a_0$ and $b_0$, respectively, such that for all $b \in B$, the mappings $a \mapsto F(a,b)$ defined on $A$ are injective.
    \end{enumerate}
    \end{thrm}
    For the sake of completeness, it's worth noting that the Kumagai-Jittorntrum implicit function theorem can be applied to control systems of form \eqref{control sys} whose domain are neighborhoods of points in $\xR^k \times \xR^\ell$. This is due to the fact that a neighborhood of a point in $\xR^p$ contains an open set homeomorphic to $\xR^p$ itself.
    
    Simple examples, such as $\dot{x}=x + u^3$, show that local asymptotic stabilizability by a composition operator with a smooth symbol is still possible in certain cases where local asymptotic stabilizability by a continuously differentiable feedback law is not.  Let's use the Kumagai-Jittorntrum implicit function theorem again to investigate this phenomenon. We will need a criterion suitably formulated for compositions to do this.
    \begin{crtrn}\label{hyp 2}
    There exists an open neighborhood of the origin $\mathcal{N} \subseteq \mathcal{X}$ and an auxiliary mapping $\beta \colon \mathcal{N} \rightarrow \mathcal{X}\times \mathcal{U}$ satisfying $\beta(0)=(0,0)$ such that the composition $f \circ \beta \colon  \mathcal{N}\rightarrow \xR^n$ is continuous and injective.
    \end{crtrn}
    The connection between Criterion~\Rref{hyp 2} and the Kumagai-Jittorntrum implicit function theorem is clear.  The first theorem here shows how this may be applied to stabilization by composition operators.  As in Theorem~\Rref{discontinuous composition stabilizability theorem}, we stress that there is \textit{no guarantee} that the symbol of the composition operator in the following result need be continuous. 
    \begin{thrm}\label{sufficiency of hyp 2}
    If Criterion~\Rref{hyp 2} holds for the system \eqref{control sys}, then \eqref{control sys} is locally exponentially stabilizable by means of a composition operator with a stationary symbol.
    \end{thrm}
    \begin{proof}
    Let $g \in \mathrm{C}(\mathcal{X},\xR^n)$ be any vector field such that $\dot{x}=g(x)$ is locally exponentially stable around the origin. Observe that if there exists an open neighborhood of the origin $\mathcal{N} \subseteq \mathcal{X}$ and a mapping $\beta \colon \mathcal{N} \rightarrow \mathcal{X} \times \mathcal{U}$ such that the mappings $y \mapsto f\big(\beta(y)\big)$ defined on $\mathcal{N}$ are continuous and injective, then given any open neighborhood of the origin $\mathcal{O}\subseteq \mathcal{X}$, the mappings $y \mapsto g(x)-f\big(\beta(y)\big)$ defined on $\mathcal{N}$ are also continuous and injective for each fixed $x \in \mathcal{O}$. Theorem~\Rref{kumagai jittorntrum implicit function theorem} tells us that there exist open neighborhoods of the origin $\mathcal{O}_0,\mathcal{N}_0 \subseteq \mathcal{X}$ and a unique continuous mapping $w \colon \mathcal{O}_0\rightarrow \mathcal{N}_0$ such that $g(x)-f\big(\beta\big(w(x)\big)\big) = 0$. Since the system $\dot{x}=g(x)$ is locally exponentially stable, and since taking $h(x)=\beta\big(w(x)\big)$ yields $g(x)=(T_hf)(x)$, the claimed local exponential stabilizability follows.  
    \end{proof}
    Notice that if the mapping $y \mapsto f(\beta(y))$ is injective, then this dictates that the mapping $y \mapsto \beta(y)$ must also be injective.  We can certainly send an open set in $\xR^n$ into a (necessarily not open) subset of $\xR^n \times \xR^m$ continuously and injectively without any trouble, so this does not preclude the possibility that one might be able to choose $\beta$ to be continuous.
    
    Along these lines, we get the next lemma as an immediate consequence of the invariance of domain.
    \begin{lmm}\label{homeo mappings 2}
    Suppose that Criterion~\Rref{hyp 2} holds for the system \eqref{control sys}. Then the image set $f\big(\beta(\mathcal{N})\big)$ is an open neighborhood of the origin, and the composition $f\circ \beta$ induces a homeomorphism $g \colon \mathcal{N} \rightarrow f\big(\beta(\mathcal{N})\big)$ such that $\iota \circ g = f \circ \beta$, where $\iota$ denotes the inclusion map $\iota \colon f\big(\beta(\mathcal{N})\big)\rightarrow \xR^n$.
    \end{lmm}
    While Theorem~\Rref{sufficiency of hyp 2} yields the existence of a stabilizing composition operator, the lack of any guarantee of continuity is undesirable.  From this, the question arises: \textit{when does a continuous auxiliary mapping satisfying Criterion~\Rref{hyp 2} exist?}  The answer is: \textit{when there exists a stabilizing composition operator with a continuous stationary symbol}.  Before showing this, the reader should recall that the continuous extension of Brockett's theorem (i.e. Theorem~\Rref{brocketts - continuous extension}) requires the uniqueness condition given in Criterion~\Rref{uniqueness criterion}.  As it should be, the following result requires the composition stabilizability uniqueness condition given in Criterion~\Rref{uniqueness criterion for composition operators} in the same fashion.
    \begin{lmm}\label{cont stabilizing comp op gives continuous hyp 2}
    Let $T_h$ be a composition operator with a continuous stationary symbol locally asymptotically stabilizing \eqref{control sys}.  Then, $\restr{T_hf}{\mathcal{O}}$ is injective for some open neighborhood of the origin $\mathcal{O}$ and $\restr{h}{\mathcal{O}}$ is a continuous function satisfying Criterion~\Rref{hyp 2}.
    \end{lmm}
    \begin{proof}
    If $\dot{x}=T_hf(x)$ has the origin as a locally asymptotically stable equilibrium, then by our assumption in Criterion~\Rref{uniqueness criterion for composition operators}, the trajectories of the system are unique when initialized in some open neighborhood $\mathcal{N}$ of the origin.  Fix $y$ in $\mathcal{N}$, choose some $T>0$, and note that $\mathcal{N}\times(\mathcal{N}\times (-T,\infty))$ is an open neighborhood of the origin in $\xR^n \times (\xR^n \times \xR)$ such that the function $F \colon \mathcal{N} \times \big(\mathcal{N}\times(-T,\infty)\big) \rightarrow \xR^n$ given by $F\big(x,(y,t)\big) = T_hf(x)-\dot{\Phi_{y}}(t+T)$ satisfies $F\big(0,(0,0)\big) = 0$ and is such that for all $(y,t) \in \mathcal{N}\times (-T,\infty)$, the equation $F\big(x,(y,t)\big)=0$ has a unique solution $x = w(y,t)$ given by the continuous mapping $w(y,t)=\Phi_{y}(t+T)$.  By the Kumagai-Jittorntrum Implicit Function Theorem, this then implies that for all $(y,t)$ in a neighborhood of $(0,0)$ contained in $\mathcal{N}\times (-T,\infty)$, the mappings $x \mapsto F\big(x,(y,t)\big)$ defined on some open neighborhood of the origin are injective.  That is, for all sufficiently small $(y,t)$, the mappings $x \mapsto T_hf(x)-T_hf\big(\Phi_{y}(t)\big)$ defined on an open neighborhood of the origin $\mathcal{O}\subseteq \mathcal{N}$ are injective. Correspondingly, taking $(y,t)=(0,0)$, it follows that the mapping $T_hf(x)-\dot{\Phi_{0}}(T) = T_hf(x)$ defined on $\mathcal{O}$ is injective, and $\restr{h}{\mathcal{O}}$ is a continuous function satisfying Criterion~\Rref{hyp 2}.
    \end{proof}
    Naturally, the reverse is obviously true as well.  That is, for stabilizability by a composition operator with a \textit{continuous} stationary symbol, the existence of a \textit{continuous} auxiliary function $\alpha$ satisfying Criterion~\Rref{hyp 2} is sufficient.
    \begin{lmm}\label{continuous function satisfies hyp 2 implies continuous right inverse lemma} Suppose that there exists a continuous mapping $\beta \colon \mathcal{N} \rightarrow \mathcal{X}\times \mathcal{U}$ which satisfies Criterion~\Rref{hyp 2} for the system \eqref{control sys}. Then the vector field $f(x)$ is open at the origin and the surjection $\tilde{f} \colon \beta(N)\rightarrow f(\beta(\mathcal{N}))$ provided by Lemma~\Rref{surjective restriction} has a continuous stationary right inverse given by $\alpha = \beta \circ g^{-1}$, where $g \colon \mathcal{N}\rightarrow f(\beta(\mathcal{N}))$ is a homeomorphism.
    \end{lmm}
    \begin{proof}
    If $f$ satisfies Criterion~\Rref{hyp 2} for some continuous mapping $\beta$, then there exists an open neighborhood of the origin $\mathcal{N} \subseteq \mathcal{X}$ and an auxiliary mapping $\beta  \colon  \mathcal{N} \rightarrow \mathcal{X}\times \mathcal{U}$ such that the composition $f\circ\beta \colon \mathcal{N}\rightarrow\xR^n$ is continuous and injective. By Lemma~\Rref{homeo mappings 2}, this composition induces a homeomorphism $g \colon \mathcal{N} \rightarrow f\big(\beta(\mathcal{N})\big)$ such that $\iota \circ g = f \circ \beta$ with the inclusion map $\iota \colon f\big(\beta(\mathcal{N})\big)\rightarrow \xR^n$. Since $g$ is a homeomorphism, it follows that $f(\beta(\mathcal{N}))$ is open and the inverse mapping $g^{-1} \colon f\big(\beta(\mathcal{N})\big)\rightarrow \mathcal{N}$ is continuous, as is $\beta \circ g^{-1}$. Since $f \circ \beta \circ g^{-1}=\iota \circ g \circ g^{-1} = \iota$, we have that $f\circ (\beta \circ g^{-1})$ is the identity on the image set $f\big(\beta(\mathcal{N})\big)$. 
                
    Similarly, it follows from the injectivity of $f \circ \beta$ that $\beta$ must also be injective.  By its continuity and the application of the invariance of domain, this again induces a homeomorphism $k \colon \mathcal{N}\rightarrow \beta(N)$ such that $\tau \circ k = \beta$, where $\tau \colon \beta(\mathcal{N})\rightarrow \mathcal{X}\times\mathcal{U}$ denotes the corresponding inclusion map. Consequentially, $k(\cdot)$ is an open mapping and $\beta(\mathcal{N})$ is a neighborhood of the origin in the subspace topology. As before, since $k(\cdot)$ is a homeomorphism, it follows that $k^{-1} \colon \beta(\mathcal{N})\rightarrow \mathcal{N}$ is continuous together with $\beta \circ k^{-1}$. Since $\beta \circ k^{-1}=\tau \circ k \circ k^{-1} = \tau$, we get
    $$
    \iota \circ g = f\circ \beta = f \circ \tau \circ k
    $$
    and thus $\iota \circ g \circ k^{-1}=f\circ \tau = \restr{f}{\beta(\mathcal{N})}$. Let now $\mathcal{O}\subseteq \mathcal{X}\times\mathcal{U}$ be an arbitrary neighborhood of the origin. Since $\beta(N)$ is a neighborhood of the origin in the subspace topology, we conclude that $\beta(\mathcal{N})\cap \mathcal{O} \neq \emptyset$, and therefore
    $$
    \big(\iota \circ g \circ k^{-1}\big)\big(\mathcal{O}\cap \beta(\mathcal{N})\big)=\restr{f}{\beta(\mathcal{N})}\big(\mathcal{O}\cap \beta(\mathcal{N})\big)\subseteq f(\mathcal{O}),
    $$
    which tells us in turn that the mapping $f$ is open at the origin.
        
    Finally, taking $\tilde{f} \colon \beta(\mathcal{N}) \rightarrow f\big(\beta(\mathcal{N})\big)$ to be the continuous surjection such that $\iota \circ \tilde{f}=\restr{f}{\beta(\mathcal{N})}$ (as is provided by Lemma~\Rref{surjective restriction}), observe that 
    $$
    \iota \circ \tilde{f} \circ (\beta \circ g^{-1})=\restr{f}{\beta(\mathcal{N})} \circ (\beta \circ g^{-1})=f \circ \beta \circ g^{-1} = \iota,
    $$
    which verifies that $\tilde{f}$ has a continuous right inverse given by $\alpha = \beta \circ g^{-1}$.
    \end{proof}
    It's indeed remarkable how easy stability is to achieve via composition operators (on a theoretical level, at least).  Namely, since Lemma~\Rref{continuous function satisfies hyp 2 implies continuous right inverse lemma} tells us that a continuous function satisfying Criterion~\Rref{hyp 2} is sufficient for stabilizability by means of composition operators, and Lemma~\Rref{cont stabilizing comp op gives continuous hyp 2} verifies the necessity, it follows that stabilizability by a composition operator with a continuous stationary symbol is equivalent to the existence of a continuous function satisfying Criterion~\Rref{hyp 2}.  
    
    Additionally, since Lemma~\Rref{continuous function satisfies hyp 2 implies continuous right inverse lemma} yields a continuous right inverse for such a system, it is natural to ask whether the converse statement holds as well.  It indeed does, as shown in the next lemma.    
    \begin{lmm}\label{continuous right inverse implies continuous function satisfies hyp 2 lemma} Let $f \in \mathrm{C}(\mathcal{X}\times\mathcal{U},\xR^n)$ be open at the origin and satisfy $f(0,0)=0$. Suppose that the surjection $\tilde{f} \colon X_0\rightarrow Y_0$ provided by Lemma~\Rref{surjective restriction} was chosen such that $X_0$ and $Y_0$ are open (as is always possible) and such that $\tilde{f}$ has a continuous stationary right inverse $\alpha \colon Y_0\rightarrow X_0$. Then, for every open neighborhood of the origin $\mathcal{N}\subseteq \mathcal{X}$ homeomorphic to $Y_0$ and every homeomorphism $g \colon \mathcal{N}\rightarrow Y_0$, the composition $\beta = \alpha \circ g$ is continuous and satisfies Criterion~\Rref{hyp 2}.
    \end{lmm}
    \begin{proof}
    Suppose $X_0$ and $Y_0$ are chosen to be open and such that the surjection $\tilde{f} \colon X_0\rightarrow Y_0$ has a continuous stationary right inverse $\alpha \colon Y_0\rightarrow X_0$.  Let $\mathcal{N}\subseteq \mathcal{X}$ be an open neighborhood of the origin homeomorphic to $Y_0$.  By the invariance of domain and the assumption that $X_0$ is taken to be open, at least one such neighborhood exists. Let $g \colon \mathcal{N}\rightarrow Y_0$ be any homeomorphism.  Then, since both $\alpha$ and $g$ must be injective, the composition $\beta = \alpha \circ g$ is injective as well. Thus we get
    $$
    f \circ (\alpha \circ g)=\iota \circ \tilde{f} \circ \alpha \circ g=\iota \circ g
    $$
    via the corresponding inclusion map $\iota \colon Y_0\rightarrow \xR^n$. Since  $\iota\circ g$ is injective and continuous, the result follows.
    \end{proof}
    Therefore, we have that stabilizability by means of a composition operator with a continuous stationary symbol is \textit{equivalent} to the existence of a continuous mapping $\beta$ satisfying Criterion~\Rref{hyp 2}, which is \textit{equivalent} to the existence of a continuous right inverse $\alpha$ on a neighborhood of the origin for $f$.  Indeed, the two are related by homeomorphisms.  While this may initially seem like a somewhat strong condition, the results of Theorem~\Rref{hautus for comp stab implies continuous right inverse} and Theorem~\Rref{discontinuous composition stabilizability theorem} suggest that it is the correct one to expect.  In light of this, we end the section with the following theorem summarizing these results.  
    \begin{thrm}\label{coronconjecture}
    The system \eqref{control sys} is locally asymptotically stabilizable by means of a composition operator $T_h$ with a continuous stationary symbol $h \colon \mathcal{O}\rightarrow \mathcal{X}\times\mathcal{U}$ if and only if there exists a stationary local section of $f$ near the origin (or, equivalently, a continuous mapping satisfying Criterion~\Rref{hyp 2}).  Moreover, such a composition operator $T_h$ can always be chosen so that this stability is exponential.
    \end{thrm}
    Having now resolved a large part of Problem~\Rref{question 1} posed in the introduction via the results of this section and Sections~\Rref{Brockett type} and \Rref{Hautus for comp stab}, we move on to the case of feedback laws.
\section{Stabilizability by Feedback Laws}
\label{kumagai ift section chapter 3}
    In this section, we will apply the composition stabilizability results of the previous sections to resolve the problem of asymptotic stabilizability by solving Problem~\Rref{question 2} as posed in the introduction.  That is, we will reduce the problem of \textit{stabilizability} to the problem of \textit{stability}.
\subsection{Stabilizability and Stability of Local Sections}\label{stabilizability by feedback laws subsection}
    To begin our path towards answering the question posed in Problem~\Rref{question 2}, we need a preliminary result that is not as useful for direct computation.
    \begin{thrm}\label{feedback stabilizable iff change of coordinates comp stabilizable}
    For $k \geq 0$, the system \eqref{control sys} is locally asymptotically (resp. exponentially) stabilizable by a $\xCn{k}$ stationary feedback law $u(x)$ if and only if there is a composition operator $T_h$ with a $\xCn{k}$ stationary symbol such that $h_1 := \restr{(\textrm{proj}_1 \circ h)}{\mathcal{O}}$ is a diffeomorphism for some open neighborhood of the origin $\mathcal{O}$ and $\dot{x}=\restr{J_{h_1}}{x}^{-1}T_hf(x)$ has the origin as a locally asymptotically (resp. exponentially) stable equilibrium.
    
    Moreover, if the system \eqref{control sys} is locally asymptotically (resp. exponentially) stabilizable by $\xCn{k}$ stationary feedback laws, then every such stabilizing feedback $u$ is of the form $u = \textrm{proj}_2 \circ h \circ (h_1)^{-1}$ for some composition operator $T_h$ with a $\xCn{k}$ stationary symbol satisfying the conditions above.
    \end{thrm}
    \begin{proof}
    In the forward direction, suppose $T_h$ is a composition operator with a $\xCn{k}$ stationary symbol locally asymptotically (resp. exponentially) stabilizing \eqref{control sys} such that $h$ is induced by a $\xCn{k}$ stationary feedback law (i.e. $h(x) = \big(x,u(x)\big)$ for some $\xCn{k}$ stationary feedback law $u$).  Then it is necessary that $(\textrm{proj}_1 \circ h)(x) = x$.  So, $h_1 := \restr{\textrm{proj}_1 \circ h}{\mathcal{O}}$ is a  diffeomorphism and $I = \restr{J_{h_1}}{x}$ for all $x$ in any open neighborhood of the origin $\mathcal{O}$ contained in the domain of $\textrm{proj}_1\circ h$.  It is straightforward to see that $\dot{x}=\restr{J_{h_1}}{x}^{-1}T_hf(x) = \restr{T_hf}{\mathcal{O}}(x)$ is locally asymptotically (resp. exponentially) stable as well, and $\textrm{proj}_2 \circ h \circ (h_1)^{-1} = u \circ \textrm{id} = u$
    
    In the reverse direction, suppose $T_h$ is composition operator with a $\xCn{k}$ stationary symbol such that $\restr{(\textrm{proj}_1 \circ h)}{\mathcal{O}}$ is a diffeomorphism for some open neighborhood of the origin $\mathcal{O}$ and $\dot{x}=\restr{J_{\textrm{proj}_1 \circ h}}{x}^{-1}T_hf(x)$ has the origin as a locally asymptotically (resp. exponentially) stable equilibrium.  Then, writing $h_2(x) = (\restr{\textrm{proj}_2\circ h}{\mathcal{O}})(x)$, define $g(x) := f\big(x,(h_2 \circ h_1^{-1})(x)\big)$ and $G(x) := \restr{J_{h_1}}{x}^{-1}T_hf(x)$.  Notice that
    $$
    G(x) = \restr{J_{h_1}}{x}^{-1}g\big(h_1(x)\big).
    $$
    So, the change-of-coordinates $y=h_1(x)$ transforms $\dot{y} = g(y)$ into $\dot{x}=G(x)$.  So, $\dot{x} = g(x) = f\big(x,(h_2 \circ h_1^{-1})(x)\big)$ has the origin as a locally asymptotically (resp. exponentially) stable equilibrium, and $u = h_2 \circ h_1^{-1}$ is a $\xCn{k}$ stationary feedback law which asymptotically (resp. exponentially) stabilizes \eqref{control sys}.
    \end{proof}
    Theorem~\Rref{feedback stabilizable iff change of coordinates comp stabilizable} is directly a consequence of a very straightforward change-of-coordinates, and is not reliant on any of the previous sections' work. The approach described by Theorem~\Rref{feedback stabilizable iff change of coordinates comp stabilizable}, however, will be of central importance in what follows.  When combined with the results of previous sections, this will indeed allow us to generate a complete characterization of stabilizability by feedback laws.  
    
    Before continuing, let us illustrate.  The existing literature leaves even some very simple systems, which are stabilizable by continuous feedback laws, largely unaddressed on anything other than a case-by-case basis.  To get a sense for Theorem~\Rref{feedback stabilizable iff change of coordinates comp stabilizable}, let's take a look at one such system.  For context, this example which was used originally by Sontag in \cite[p. 239]{sontag2} to demonstrate, in particular, that the real `sticking point' which prevents certain systems from having $\xCone$ stabilizing controls is the behavior of the control at the equilibrium.  This provides an instance of a system that \textit{does not} have a stabilizable linearization, but \textit{does} have an exponentially stabilizing control that (necessarily) fails to be $\xCone$.
    \begin{xmpl}\label{ex 1 chapter 3}
    Consider the system \eqref{control sys} with $f(x,u):=x+u^3$. By $A_f = 1$ and $B_f = 0$, the linearization of the system is given by $\dot{x}=x$. Of course, this is not a stabilizable system (the Hautus lemma shows it immediately since $\textrm{rank}\left[\begin{array}{c|c}\lambda I - A_f & B_f\end{array}\right] = 0$ for the sole eigenvalue $\lambda=1$ of $A_f$). However, it follows by a routine calculation that the feedback law $u(x)=(-2x)^{1/3}$ globally exponentially stabilizes the system, and it is indeed continuous while not $\xCone$ at the origin.  
            
    Observe now that the function $h \colon \xR \rightarrow \xR^2$ defined by $h(x) := \left(-\frac{x}{2},\sqrt[3]{x}\right )$ is continuous and has $\textrm{proj}_1 \circ h = -\frac{x}{2}$, which is certainly a $\textrm{C}^\infty$ diffeomorphism.  Since $T_hf(x) = -\frac{x}{2}+x = \frac{x}{2}$ and $\restr{J_{\textrm{proj}_1\circ h}}{x}^{-1} = -2$, we get that 
    $$
    \dot{x}=\restr{J_{\textrm{proj}_1\circ h}}{x}^{-1}T_hf(x) = -x.
    $$
    This system is, of course, locally exponentially stable and (as expected) the stabilizing control produced by Theorem~\Rref{feedback stabilizable iff change of coordinates comp stabilizable} is  $u(x)=\big(\textrm{proj}_2 \circ h \circ (\textrm{proj}_1\circ h)^{-1}\big)(x)=(-2x)^{1/3}$.
    \end{xmpl}
    
    To employ the approach in Theorem~\Rref{feedback stabilizable iff change of coordinates comp stabilizable}, we first require a useful fact:
    \begin{thrm}\label{system under comp op is homeomorphism}
    Suppose \eqref{control sys} is locally asymptotically stabilizable by a composition operator $T_h$ with a $\xCn{k}$ stationary symbol for $k \geq 0$.  Then, $\restr{T_hf}{\mathcal{N}}$ is a homeomorphism for any open neighborhood of the origin $\mathcal{N}$ contained in the domain of $h$.  
    
    If, in addition, $T_h$ exponentially stabilizes \eqref{control sys} and is such that $T_hf$ is $\xCn{\ell}$ for $\ell \geq 1$, then $\restr{T_hf}{\mathcal{V}}$ is a $\xCn{\ell}$ diffeomorphism for some open connected neighborhood of the origin $\mathcal{V}$ contained in the domain of $h$.
    \end{thrm}
    \begin{proof}
    Let $T_h$ be a composition operator with $\xCn{k}$ stationary symbol which locally asymptotically stabilizes \eqref{control sys}.  Then, via our assumption in Criterion~\Rref{uniqueness criterion for composition operators}, $T_hf$ is continuous and $\dot{x}=T_hf(x)$ can be locally asymptotically stabilized by means of a composition operator with a continuous stationary symbol--namely, $T_x$. By Theorem~\Rref{coronconjecture}, it follows that $\restr{T_hf}{\mathcal{N}}$ has a continuous stationary right inverse $\alpha:\mathcal{O} \rightarrow \mathcal{N}$ defined on an open neighborhood of the origin $\mathcal{O}$.  Since $\restr{T_hf}{\mathcal{N}} \circ \alpha = \iota$, where $\iota:\mathcal{O}\rightarrow \xR^n$ is the inclusion map, it follows that $\alpha$ is injective.  By the invariance of domain, $\alpha$ is a homeomorphism and has a continuous inverse $\alpha^{-1}:\mathcal{N} \rightarrow \mathcal{O}$.  Hence, we have $\restr{T_hf}{\mathcal{N}} = \restr{T_hf}{\mathcal{N}} \circ \alpha \circ \alpha^{-1} = \iota \circ \alpha^{-1}$, which is certainly an injection.  By another application of the invariance of domain, $\restr{T_hf}{\mathcal{N}}$ is a homeomorphism.
    
    If, in addition, $T_hf$ is $\xCn{\ell}$ and has the origin as a locally exponentially stable equilibrium, then treating $\dot{x}=T_hf(x)$ as a control system with trivial dependence on the control variable, Corollary~\Rref{hautus-zabczyk} implies that $\restr{J_{T_hf}}{0} \in \xGL_n(\xR)$, and the constant rank theorem yields there is a neighborhood of the origin such that $\restr{J_{T_hf}}{x} \in \xGL_n(\xR)$ for each $x$ in this neighborhood.  Take $\mathcal{V}$ to be a connected open neighborhood of the origin (any sufficiently small open ball about the origin will suffice) contained in this region.  Then, since $T_hf$ is $\xCn{\ell}$, $\restr{T_hf}{\mathcal{V}}$ is $\xCn{\ell}$ and Theorem~\Rref{hautus for comp stab implies continuous right inverse} yields that the right inverse $\alpha$ of $\restr{T_hf}{\mathcal{V}}$ can be taken to be $\xCn{\ell}$.  By the above argument, $\restr{T_hf}{\mathcal{V}} = \iota \circ \alpha^{-1}$, so it follows from the fact that $\alpha$ is a homeomorphism that $\restr{T_hf}{\mathcal{V}}$ is closed.  Notice as well that for every $y \in \mathcal{O}$, we have $\restr{T_hf}{\mathcal{V}}^{-1}(y) = (\iota \circ \alpha^{-1})^{-1}(y) = \alpha\big(\iota^{-1}(y)\big) = \alpha(y)$. Similarly, for every $y \in \xR^n \setminus \mathcal{O}$, we have $\restr{T_hf}{\mathcal{V}}^{-1}(y) = \emptyset$.  So, $\restr{T_hf}{\mathcal{V}}$ is closed with compact fibers, and since $\xR^n$ is locally compact and Hausdorff, it follows that $\restr{T_hf}{\mathcal{V}}$ is proper.  Therefore, $\restr{T_hf}{\mathcal{V}}$ has a bijective derivative for each $x \in \mathcal{V}$ and is a proper differentiable map between an open subset $\mathcal{V} \subseteq \xR^n$ and the open simply connected subset $\xR^n\subseteq \xR^n$.  By the Hadamard-Caccioppoli theorem, it then follows that $\restr{T_hf}{\mathcal{V}}$ is a diffeomorphism.
    \end{proof}
    We can use this to great effect in the situations where we seek an exponentially stabilizing feedback law $u(x)$ such that $f\big(\cdot,u(\cdot)\big)$ is, at least, of class $\xCn{1}$.
    \begin{thrm}\label{stabilizable by controls if eigenvalues of right inverse are good}
    For $k \geq 0$ and $\ell \geq \textrm{max}\left\{1,k\right\}$, the system \eqref{control sys} is locally exponentially stabilizable by means of a $\xCn{k}$ stationary feedback law $u$ such that $f\big(\cdot,u(\cdot)\big)$ is $\xCn{\ell}$ if and only if $u = \textrm{proj}_2\circ \alpha \circ (\textrm{proj}_1 \circ \alpha)^{-1}$ for some $\xCn{k}$ stationary right inverse $\alpha$ of $f$ defined on an open neighborhood of the origin such that $\textrm{proj}_1\circ \alpha$ is a $\xCn{\ell}$ diffeomorphism and $\Lambda_+(\restr{J_{\alpha_1}}{0}^{-1}) = \emptyset$.
    \end{thrm}
    \begin{proof}
    In the forward direction, suppose \eqref{control sys} is locally exponentially stabilizable by the $\xCn{k}$ stationary feedback law $u$ inducing $w(x)=\big(x,u(x)\big)$ such that $T_wf$ is $\xCn{\ell}$ for $\ell \geq \textrm{max}\left\{1,k\right\}$.  By Theorem~\Rref{system under comp op is homeomorphism}, $\restr{T_wf}{\mathcal{N}}$ is a $\xCn{\ell}$ diffeomorphism for some open connected neighborhood of the origin $\mathcal{N}$.  So, $\alpha := w \circ \restr{T_wf}{\mathcal{N}}^{-1}$ is a $\xCn{k}$ stationary right inverse of $f$ on an open neighborhood of the origin and 
    $$
    \alpha_1:= \textrm{proj}_1 \circ \alpha = \textrm{proj}_1 \circ w \circ \restr{T_wf}{\mathcal{N}}^{-1} = \restr{T_wf}{\mathcal{N}}^{-1}
    $$
    is a $\xCn{\ell}$ diffeomorphism.  Observe as well that $(\textrm{proj}_2\circ \alpha \circ \alpha_1^{-1})(x) = u(x)$ is certainly a $\xCn{k}$ stabilizing control for \eqref{control sys}. Treating $\dot{x}=\restr{T_wf}{\mathcal{N}}(x)$ as a control system with trivial dependence on the control variable, Corollary~\Rref{hautus-zabczyk} implies that $\Lambda_+\left(\restr{J_{\restr{T_wf}{\mathcal{N}}}}{0}\right) = \emptyset$.  Since $T_wf(0)=0$, the inverse function theorem tells us that 
    $$
    \restr{J_{\alpha_1}}{0} = \restr{J_{\restr{T_wf}{\mathcal{N}}^{-1}}}{T_wf(0)} = \restr{J_{\restr{T_wf}{\mathcal{N}}}}{0}^{-1}.
    $$
    Hence, we conclude that $\Lambda_+ \left(\restr{J_{\alpha_1}}{0}^{-1} \right) = \emptyset$.
    
    In the reverse direction, suppose $f$ has a $\xCn{k}$ stationary right inverse $\alpha$ defined on a neighborhood of the origin $\mathcal{O}$ such that $\alpha_1 := \textrm{proj}_1\circ \alpha$ is a $\xCn{\ell}$ diffeomorphism for $\ell \geq \textrm{max}\left\{1,k\right\}$ and $\Lambda_+\left(\restr{J_{\alpha_1}}{0}^{-1} \right) = \emptyset$. Write $F(x) = f\big(x,\textrm{proj}_2 \circ \alpha \circ \alpha_1^{-1}(x)\big)$ and notice that $\alpha_1^{-1} \circ \alpha_1 = F \circ \alpha_1$.  Since $\alpha_1$ is a $\xCn{\ell}$ diffeomorphism, it must certainly be surjective, and as surjections are the epimorphisms in the category of sets, it follows that $F(x)= \alpha_1^{-1}$.  So, $F$ is a $\xCn{\ell}$ diffeomorphism as well.  Since $\restr{J_F}{\alpha_1(0)}=\restr{J_F}{0}=\restr{J_{\alpha_1}}{0}^{-1}$ by the inverse function theorem and the stationarity of $\alpha$, the local exponential stability of $\dot{x}=F(x)$ follows from $\Lambda_+\left(\restr{J_{\alpha_1}}{0}^{-1}\right) = \emptyset$.  Clearly, this yields that $u=\textrm{proj}_2 \circ \alpha \circ (\textrm{proj}_1 \circ \alpha)^{-1}$ is a $\xCn{k}$ stationary feedback law locally exponentially stabilizing \eqref{control sys}.
    \end{proof}
    The case for asymptotic and exponential stabilizability when $f\big(\cdot,u(\cdot)\big)$ is only $\xCn{0}$ follows similarly:
    \begin{thrm}\label{stabilizable by controls if right inverse stable}
    The system \eqref{control sys} is locally asymptotically (resp. exponentially) stabilizable by means of continuous stationary feedback laws $u$ if and only if $u = \textrm{proj}_2\circ \alpha \circ (\textrm{proj}_1 \circ \alpha)^{-1}$ for some local section $\alpha$ of $f$ defined on an open neighborhood of the origin such that $\alpha_1 := \textrm{proj}_1\circ \alpha$ is a homeomorphism and $\dot{x}=\alpha_1^{-1}(x)$ has the origin as a locally asymptotically (resp. exponentially) stable equilibrium.
    \end{thrm}
    \begin{proof}
    In the forward direction, suppose \eqref{control sys} is locally asymptotically (resp. exponentially) stabilizable by a continuous stationary feedback law $u$.  Correspondingly, this induces the stabilizing composition operator $T_w$ with $w(x)=\big(x,u(x)\big)$.  By Theorem~\Rref{system under comp op is homeomorphism}, $\restr{T_wf}{\mathcal{N}}$ is a homeomorphism for some open connected neighborhood of the origin $\mathcal{N}$.  So, $\alpha := w \circ \restr{T_wf}{\mathcal{N}}^{-1}$ is a continuous stationary right inverse of $f$ on an open neighborhood of the origin and, as before, $\alpha_1 = \restr{T_wf}{\mathcal{N}}^{-1}$ is a homeomorphism.  Observe as well that $(\textrm{proj}_2\circ \alpha \circ \alpha_1^{-1})(x) = u(x)$ is certainly a continuous stationary feedback law stabilizing \eqref{control sys}.
    
    In the reverse direction, suppose $f$ has a continuous stationary right inverse $\alpha$ defined on a neighborhood of the origin $\mathcal{O}$ such that $\alpha_1 := \textrm{proj}_1\circ \alpha$ is a homeomorphism and $\dot{x}=\alpha_1^{-1}(x)$ has the origin as a locally asymptotically (resp. exponentially) stable equilibrium.  Write $F(x) = f\big(x,\textrm{proj}_2 \circ \alpha \circ \alpha_1^{-1}(x)\big)$ and notice that $\alpha_1^{-1} \circ \alpha_1 = F \circ \alpha_1$.  Since $\alpha_1$ is a homeomorphism, it must certainly be surjective, and as surjections are the epimorphisms in the category of sets, it follows that $F(x)= \alpha_1^{-1}$.  So, the local exponential stability of $\dot{x}=F(x)$ follows from the local  asymptotic (resp. exponential) stability of $\dot{x} = \alpha_1^{-1}(x)$.  Clearly, this yields that $u=\textrm{proj}_2 \circ \alpha \circ (\textrm{proj}_1 \circ \alpha)^{-1}$ is a continuous stationary feedback law locally  asymptotically (resp. exponentially) stabilizing \eqref{control sys}.
    \end{proof}
    With the results of Theorem~\Rref{stabilizable by controls if eigenvalues of right inverse are good} and Theorem~\Rref{stabilizable by controls if right inverse stable} in hand, we can answer the question of stabilizability by feedback laws quite nicely.
    \begin{thrm}\label{stabilizable so that infinitely differentiable characterization}
     For $k \geq 0$, the system \eqref{control sys} is locally asymptotically (resp. exponentially) stabilizable by means of a $\xCn{k}$ stationary feedback law $u$ if and only if there exists a $\xCn{k}$ stationary local section $\alpha$ of $f$ near the origin such that $\alpha_1 := \textrm{proj}_1\circ \alpha$ is a homeomorphism (resp. diffeomorphism) and $\dot{x}=\alpha_1^{-1}(x)$ is locally asymptotically (resp. exponentially) stable.  Moreover, every such control $u$ is of the form $u =\textrm{proj}_2\circ \alpha \circ (\textrm{proj}_1 \circ \alpha)^{-1}$ for some $\alpha$ satisfying the above.
    \end{thrm}
    Broadly speaking, Theorem~\Rref{stabilizable so that infinitely differentiable characterization} ties a nice bow on the question of stabilizability in a very easy to understand way.  Speaking colloquially and in clear terms, here is a summary of what we have learned:
    \begin{itemize}
    \item Exponential stabilizability by a composition operator with $\xCn{k}$ stationary symbols is possible if and only if we locally have a $\xCn{k}$ stationary local section near the origin.  More simply, when we can change \textit{every argument} of the dynamics, stabilization at an equilibrium is possible if and only if we have a tool to \textit{change the dynamics at will}.
    \item Stabilizability by $\xCn{k}$ stationary feedback laws is possible if and only if we locally have a $\xCn{k}$ stationary local section near the origin which is invertible in the non-control variables and whose inverse induces a stable system.  More simply, when we impose that we may change \textit{only some arguments} of the dynamics, stabilization at an equilibrium is possible if and only if we have two things: (1) a tool to \textit{change the dynamics at will} and (2) a second tool to undo the effects of the first tool on the arguments we cannot change that, itself, \textit{stabilizes those unchanged arguments}.
    \end{itemize}
    Let us conclude with an example to drive home this theme.
    \begin{xmpl}\label{lastexample}
    Consider the system \eqref{control sys} where $f:\xR^2 \times \xR \rightarrow \xR^2$ is defined by:
    $$
        f(x,u) = \begin{bmatrix} 
            x_1^2 + x_2^2 + x_2 \\
            x_1x_2+x_2^2+u^3
        \end{bmatrix}
    $$
    We begin by noting that Corollary~\Rref{hautus-zabczyk} yields, by routine computation, that this system fails to be exponentially stabilizable by $\xCone$ stationary feedback laws.  It does, however, turn out to be exponentially stabilizable by $\xCzero$ stationary feedback laws, in a manner analogous to Example~\Rref{ex 1 chapter 3}.  We will employ Theorem~\Rref{stabilizable so that infinitely differentiable characterization} to this end.
    
    We will construct a right inverse for $f$.  To do so, let $\alpha_1,\alpha_2,\alpha_3:\xR^2 \rightarrow \xR$ be the (as-of-yet, unspecified) component functions of this sought right inverse $\alpha:\xR^2 \rightarrow \xR^2 \times \xR$.  When constructing $\alpha$, we must be careful to do so in such a way so as to satisfy the constraints of Theorem~\ref{stabilizable so that infinitely differentiable characterization}--namely,  $\textrm{proj}_1\circ \alpha$ must be a diffeomorphism defined on some neighborhood of the origin and $\dot{x}=(\textrm{proj}_1\circ \alpha)^{-1}(x)$ must be locally exponentially stable.   
    
    To this end, we start in a straightforward fashion and observe that the first component of $f$ yields that $\alpha_1$ and $\alpha_2$ must satisfy $\alpha_1^2+\alpha_2^2+\alpha_2 = x_1$.  Using the quadratic formula, it is easy to see that this equation will be satisfied if 
    $$
    \alpha_2=\frac{1}{2}\left(-1+\sqrt{1-4(\alpha_1^2-x_1)}\right).
    $$
    Hence, if $\textrm{proj}_1\circ\alpha$ is to be invertible, then this inverse function must satisfy $\alpha_1\big((\textrm{proj}_1\circ \alpha)^{-1} \big) = x_1$, as well as:
    \begin{align*}
        \alpha_2\big((\textrm{proj}_1\circ \alpha)^{-1} \big) &= x_2 \\ \frac{1}{2}\left(-1+\sqrt{1-4\big(x_1^2-(\textrm{proj}_1\circ \alpha)^{-1}_1\big)}\right) &= x_2 \\
        (\textrm{proj}_1\circ \alpha)^{-1}_1 &= x_1^2+x_2^2+x_2.
    \end{align*}
    Applying this information, note that the linearization of $\dot{x}=(\textrm{proj}_1\circ \alpha)^{-1}(x)$ would then be of the form
    $$
        \restr{J_{(\textrm{proj}_1 \circ \alpha)^{-1}}}{(0,0)} = \begin{bmatrix}
        0 &1 \\
        \frac{\partial (\textrm{proj}_1 \circ \alpha)^{-1}_2}{\partial x_1} & \frac{\partial (\textrm{proj}_1 \circ \alpha)^{-1}_2}{\partial x_2} 
        \end{bmatrix}.
    $$
    So, let us take $(\textrm{proj}_1 \circ \alpha)^{-1}_2 = -2x_2-\frac{1}{2}x_1$, producing eigenvalues of $\lambda = 1/2 (-2 \pm \sqrt{2}) < 0$.  Correspondingly, $(\textrm{proj}_1 \circ \alpha)^{-1}$ is now well-defined and the system $\dot{x}=(\textrm{proj}_1 \circ \alpha)^{-1}(x)$ is locally exponentially stable.  Via an application of the inverse function theorem, we reach $\textrm{proj}_1 \circ \alpha$ as the inverse of this function near the origin, thus specifying $\alpha_1$ and $\alpha_2$.
    
    To conclude, we complete our construction of $\alpha$ by solving $\alpha_1\alpha_2+\alpha_2^2+\alpha_3^3=x_2$ for $\alpha_3$, producing $\alpha_3 = \sqrt[3]{x_2-\alpha_1\alpha_2-\alpha_2^2}$.  By Theorem~\Rref{stabilizable so that infinitely differentiable characterization}, the $\xCzero$ stationary feedback law
    
    $$
        u(x) = \sqrt[3]{(\textrm{proj}_1 \circ \alpha)^{-1}_2-x_1x_2-x_2^2} = \sqrt[3]{-2x_2-\frac{1}{2}x_1-x_1x_2-x_2^2}
    $$
    
    \noindent is a stabilizing control. To confirm this, observe that $
    f\big(x,u(x)\big) = \begin{bmatrix}
    x_1^2+x_2^2+x_2 \\
    -2x_2-\frac{1}{2}x_1
    \end{bmatrix}.
    $ Linearizing the right-hand side, we produce $
    \restr{J_{f\circ (\textrm{id},u)}}{(0,0)} = \begin{bmatrix} 
        0 & 1 \\
        -\frac{1}{2} & -2
    \end{bmatrix},
    $
    the eigenvalues of which are given by $\lambda = 1/2 (-2 \pm \sqrt{2}) < 0$ and are, as they should be, the same as the eigenvalues of the linearization of $\dot{x}=(\textrm{proj}_1 \circ \alpha)^{-1}(x)$.
    
    \end{xmpl}
    Before finishing, it is worth providing a some words here regarding the use of Theorem~\Rref{stabilizable so that infinitely differentiable characterization} in practice.  In particular, not all examples will work out as nicely as the one above.  Indeed, even with the characterization given by Theorem~\Rref{stabilizable so that infinitely differentiable characterization}, explicit constructions of stabilizing controls is still, quite often, very challenging.  In large part, this is due to the fact that the local sections for the vector field $f$ in \eqref{control sys}, should they exist, might not be possible to express in a closed form.  In some situations where this occurs, this is not necessarily a problem--as illustrated by the technique employed in Example~\Rref{lastexample} above, notice that we did not need to actually compute $\textrm{proj}_1 \circ \alpha$ directly to obtain our stabilizing feedback law.  So, the actual computation involved in constructing a stabilizing feedback law can, in some particular circumstances, be done without the explicit construction of a local section (allowing us to occasionally overcome the aforementioned difficulty).  Unfortunately, such tricks do not always resolve similar issues that may arise. Even if a local section of $f$ can be expressed in a closed form, this does not guarantee that $(\textrm{proj}_1 \circ \alpha)^{-1}$ enjoys the same property.  Similar problems persist in the reverse as well--it may happen that $(\textrm{proj}_1 \circ \alpha)^{-1}$ is expressible in a closed form while $\textrm{proj}_2 \circ \alpha$ is not.  More confusingly, a stabilizing control may be expressible in closed form even when there exists no local section $\alpha$ such that both $(\textrm{proj}_1 \circ \alpha)^{-1}$ and $\textrm{proj}_2\circ \alpha$ share the same property, since the composing two functions that cannot be expressed in closed form does not guarantee that the same holds for the result.  More frustratingly, it is important to note that we cannot simply employ some approximation technique in computing the control and ignore this challenge, as Theorem~\Rref{stabilizable so that infinitely differentiable characterization} tells us that \textit{every possible stabilizing feedback} must be of the form $u =\textrm{proj}_2\circ \alpha \circ (\textrm{proj}_1 \circ \alpha)^{-1}$ for some local section $\alpha$ satsisfying the conditions of the theorem.  As such, one should expect to occasionally encounter systems which are stabilizable, but whose stabilizing controls are beyond our practical reach.
\section{Conclusion and Future Work}\label{conclusion section chapter 3}
    In this paper, we have demonstrated that many of the most well-known classical theorems regarding stabilizability have somewhat cleaner forms in the composition operator context.  We also, hopefully, were able to convince the reader that extracting properties about the composition operators which stabilize a given system \eqref{control sys} tends to be, in general, relatively easy (at least, in comparison with extracting information about stabilizing feedback control laws).  By employing these composition-stabilizability results in the classical feedback-stabilizability context, we have shown that the \textit{stabilizability} of a control systems is equivalent to the \textit{stability} of an associated system.  That is, we reduce the question of \textit{stabilizability} to that of \textit{stability}.

    There is a lot of potential for further work characterizing stabilizability via composition operators.  In particular, completing the distinction between asymptotic and exponential stabilization via composition operators--that is, closing the gap between Theorem~\Rref{hautus for comp stab implies continuous right inverse} and Theorem~\Rref{coronconjecture}--could be a worthwhile endeavor.  Additionally, conjecture~\Rref{brockett's sufficient conjecture} would be interesting to investigate further.
        
    From the characterization of stabilizing controls and stabilizable systems, there seem to be a variety of potential avenues for further study.  For example, Coron's result from \cite{coron2} give us a stronger necessary condition than Brockett's theorem by establishing a relationship between the existence of continuous feedback laws and the existence of an isomorphism between certain singular homology (or, equivalently, stable homotopy) groups associated to the system. Refining arguments along homological lines could prove fruitful, in light of the characterizations provided by Theorem~\Rref{stabilizable so that infinitely differentiable characterization}.  Additionally, since stability can be completely characterized via Lyapunov functions (as Zubov's theorem shows; see \cite{zubov1964methods}), investigations of Lyapunov-theoretic connections in the context of the results presented here could provide a bridge between the existence of local sections and the existence of Lyapunov functions.
        
    It also would be worth pursuing variational-analytic connections. Namely, extending the established characterizations of local asymptotic and exponential stabilizability to the case of \textit{set-valued} differential inclusions could shed some light on the topic of stabilizability in a more general sense (particularly, by removing the uniqueness criteria imposed throughout this paper). As mentioned and exploited in \cite{gjkm,christopherson2019}, variational analysis achieves complete characterizations of linear openness (see, e.g., \cite{m18} and the references therein), a notion which might potential to serve as the counterpart to Brockett's openness property from Definition~\Rref{open} for general nonsmooth mappings and multifunctions (indeed, Brockett's property has already been investigated in this context to some degree \cite{ryan1994brockett}). So, much of the groundwork for further developments has already been completed. The real challenge would be implementing the classical, known results regarding feedback stabilizability (or, better, stabilizability via composition operators) in this much wilder context.


\end{document}